\documentclass[11pt,reqno, a4paper]{amsart}
\usepackage{amsfonts}
\usepackage{amssymb}
\usepackage{txfonts}
\usepackage{mathrsfs}
\usepackage{bbm}
\usepackage{enumerate}

\newtheorem{thm}{Theorem}[section]
\newtheorem{lem}[thm]{Lemma}
\newtheorem{prp}[thm]{Proposition}

\theoremstyle{definition}

\newtheorem{rmk}{Remark}[section]
\numberwithin{equation}{section}

 \DeclareMathOperator{\RE}{Re}

  \newcommand{\tri}{\triangle}
 \newcommand{\eps}{\varepsilon}
 \newcommand{\vpi}{\varphi}
 \newcommand{\p}{\mathcal{P}}
 
 \newcommand{\D}{\mathcal{D}}
 \newcommand{\Real}{\mathbb{R}}
 
 \newcommand{\Natural}{\mathbb{N}}
 \newcommand{\nrho}{\inner{\frac{N}{\rho}}}
 \newcommand{\norm}[1]{\left\Vert#1\right\Vert}
 \newcommand{\bignorm}[1]{\big\Vert#1\big\Vert}
 \newcommand{\abs}[1]{\left\vert#1\right\vert}
 \newcommand{\set}[1]{\left\{#1\right\}}
 
 \newcommand{\bigcom}[1]{\big[#1\big]}
 \newcommand{\bigset}[1]{\big\{#1\big\}}
 \newcommand{\inner}[1]{\left(#1\right)}
 \newcommand{\biginner}[1]{\big(#1\big)}
 \newcommand{\Biginner}[1]{\Big(#1\Big)}
 
 \newcommand{\dol}[1]{$#1$}
 \newcommand{\mc}[1]{{\mathcal #1}}

 \newcommand{\bb}[1]{{\mathbb #1}}
 \newcommand{\reff}[1]{(\ref{#1})}

\newcommand\cD{\mathcal D}

\newcommand\cS{{\mathcal S}}

\newcommand\NN{{\mathbb N}}
\newcommand\RR{{{\mathbb R}}}

\begin{document}

\title[The Gevrey hypoellipticity for kinetic equations]
{Gevrey hypoellipticity for a class\\
of kinetic equations}
\author[H. Chen]{Hua CHEN}
\address{School of Mathematics and Statistics, Wuhan University,
Wuhan 430072, China} \email{chenhua@whu.edu.cn}
\author[W.-X. Li]{Wei-Xi LI}
\address{School of Mathematics and Statistics, Wuhan University,
Wuhan 430072, China} \email{wei-xi.li@whu.edu.cn}
\author[C.-J. Xu]{Chao-Jiang XU}
\address{Universit\'e de Rouen, UMR 6085-CNRS, Math\'ematiques,
Avenue de l'Universit\'e, \newline \indent BP.12, 76801 Saint
Etienne du Rouvray, France
\newline
\indent and\newline \indent School of mathematics, Wuhan University,
430072, Wuhan, China} \email{Chao-Jiang.Xu\@@univ-rouen.fr}

\date{2008-03-14\,\,\,  Research supported partially by the NSFC grant 10631020}
\subjclass[2000]{35 B, 35 H, 35 N}


\begin{abstract}
In this paper, we study the Gevrey regularity of weak solutions for
a class of linear and semi-linear kinetic equations, which are the
linear model of spatially inhomogeneous Boltzmann equations without
an angular cutoff.
\end{abstract}

\maketitle

\section{Introduction}\label{sect1}

In this paper, we study  the following  kinetic operator:
\begin{equation}\label{Fokker-Planck}
\mc{P}=\partial_t+v\cdot\partial_x+a(t,x,v)(-\widetilde\triangle_v)^\sigma
,\quad(t,x, v)\in{\Real\times\Real^n\times\bb R^{n}},
\end{equation}
where $ 0<\sigma < 1$,
$v\cdot\partial_x=\Sigma_{j=1}^nv_j\partial_{x_j}$, $a(t,x,v)\in
C^\infty(\bb R^{2n+1})$ and $a(t,x,v)>0$ on
${\Real\times\Real^n\times\bb R^{n}}$, the notation
$(-\widetilde\triangle_v)^\sigma$ denotes the Fourier multiplier of
symbol
$$p(\eta)=\bigset{\abs\eta^{\sigma}\omega(\eta)+\abs\eta(1-\omega(\eta))}^2,$$
 with $\omega(\eta)\in C^\infty(\Real^n)$, $0\leq\omega\leq 1$.
 Moreover, we have
 $\omega=1$ if $\abs\eta\geq 2$ and $\omega=0$ if $\abs\eta\leq1$.
Throughout the paper, we denote by $\hat u(\tau,\xi,\eta)$ the
Fourier transform of $u$ with respect to the variables $(t,x,v)$.
$\mc{P}$ is not a classical pseudo-differential operator in $\bb
R^{2n+1}$; for the coefficient in the kinetic part is not bounded in
$\bb R^{2n+1}$. When $\sigma=1$, the operator (1.1) is the so-called
Vlasov-Fokker-Planck operator (see \cite{helffer-nier,herau-nier}),
it is then a H\"ormander type operators, and we can apply the Gevrey
hypoellipticity results of M. Derridj and C. Zuily \cite{DZ} and M.
Durand \cite{Dur}, see also [5] for the optimal $G^3$-hypoelliptic
results.

As  is well known, the operator (1.1) is a linear model of the
spatially inhomogeneous Boltzmann equation without an angular cutoff
(cf. \cite{MX}). This is the main motivation for the study of the
regularizing properties of the operator (1.1) in  this paper. In the
past several years,  a lot of progress has been made in the study of
the spatially homogeneous Boltzmann equation without an angular
cutoff, (see \cite{al-1,al-2,desv-wen1,villani} and references
therein), in which the authors have proved that the singularity of
the collision cross-section  yields certain gain on the regularity
for the weak solution of the Cauchy problem in the Sobolev space
frame. That implies that there exists a $C^\infty$ smoothness effect
of the Cauchy problem for the spatially homogeneous Boltzmann
equation without an angular cutoff. The Gevrey regularity of the
local solutions has been constructed in \cite{ukai} for the initial
data having the same Gevrey regularity, and the propagation of
Gevrey regularity is proved recently in \cite{DFT}. In \cite{MUXY1},
the Gevrey smoothness effect of the Cauchy problem has been
established for the spatially homogeneous linear Boltzmann equation.
In \cite{MX2}, they  obtain the ultra-analytical effect results for
the non linear homogeneous Landau equations and inhomogeneous linear
Landau equations.

However, there is no general result for the smoothness effect of the
spatially inhomogeneous problem, which is actually related with the
regularity of the kinetic equation with its diffusion part  a
nonlinear operator in the  velocity variable $v$. Under the
singularity assumption on the collision cross section, the behavior
of the Boltzmann collision operator is similar to a fractional power
of the Laplacian $\inner{-\tri _v}^\sigma.$ In \cite{aumxy}, by
using the uncertainty principle of the micro-local analysis,  the
authors obtained $C^\infty$ regularity for the weak solution of the
linear spatially inhomogeneous Boltzmann equation without an angular
cutoff.

On the other hand, in \cite{MX}, the existence and the $C^\infty$
regularity have been proved for the solutions of the Cauchy problem
for  linear and semi-linear equations associated with the kinetic
operators (\ref{Fokker-Planck}). In this paper, we shall consider
the Gevrey regularity for such  problems.

Let us first recall the definition for the  functions in the Gevery
class. Let $U$ be an open subset of $\mathbb{R}^d$ and $1\leq
s<+\infty$, we say that $f\in G^s(U)$ if $f\in C^\infty(U)$ and for
any compact subset $K$ of $U$, there exists a constant (say Gevrey
constant of $f$) $C=C_K$, depending only on $K$ and $f$, such that
for all multi-indices $\alpha\in\mathbb{N}^d$,
\begin{eqnarray}\label{gevrey}
\|\partial^\alpha{f}\|_{L^\infty(K)} \leq
C_K^{|\alpha|+1}(\alpha!)^s.
\end{eqnarray}
If $W$ is a closed subset of $\bb{R}^d$, $G^s(W)$ denote the
restriction of $G^s(\tilde{W})$ on $W$ where $\tilde{W}$ is an open
neighborhood of $W$. The condition {\rm (\ref{gevrey})} is
equivalent to the following estimate (e.g. see \cite{CR} or
\cite{Ro}):
$$
\|\partial^\alpha{f}\|_{L^2(K)}\leq C_K^{|\alpha|+1}(|\alpha|!)^{s}.
$$

We say that an operator $P$ is $G^s$ hypoelliptic in $U$ if $u\in
\D'(U)$ and $ Pu\in G^s(U)$, then it follows that  $u\in G^s(U).$
Likewise, we say that  the operator $P$ is $C^\infty$ hypoelliptic
in $U$ if $u\in \D'(U)$ and $Pu\in C^\infty(U)$, then it follows
that  $u\in C^\infty(U).$

\vspace{0.5ex}In \cite{MX}, Morimoto-Xu  proved that the operator
(1.1) is $C^\infty$ hypoelliptic if $1/3<\sigma\leq 1$. Our first
main result of this paper is the following:

\begin{thm}\label{th2}
Let $0<\sigma<1$ and
$\delta=\max\set{\frac{\sigma}{4},~\frac{\sigma}{2}-\frac 1 6}$.
Then the operator $\mc{P}$ given by (\ref{Fokker-Planck}) is $G^s$
hypoelliptic in $\mathbb{R}^{2n+1}$ for any $s\geq \frac{2}{\delta}$
, provided the coefficient $ a(t,x,v) \in G^s(\mathbb{R}^{2n+1})$
and $a(t,x,v)>0$.
\end{thm}

Compared with what is obtained  in \cite{MX}, the result of Theorem
1.1 implies that the operator (1.1) is also $C^\infty$ hypoelliptic
in the case of $0<\sigma\leq 1/3$.

Next, we consider the following semi-linear equation:
\begin{equation}\label{++1.2}
\partial_t u+v \cdot \nabla _x u+ a
 (-\widetilde\triangle_v)^{\sigma}u=F(t, x, v;u)
\end{equation}
where $F$ is a nonlinear function of the real variables $(t, x,
v,q)$. The following is the second main result of the paper, which
implies that the weak solution of equation (1.3) has Gevrey
regularity:

\begin{thm}\label{th3}

Let $0<\sigma<1$ and
$\delta=\max\set{\frac{\sigma}{4},~\frac{\sigma}{2}-\frac 1 6}$.
Suppose that $u\in L_{loc}^{\infty}(\bb R^{2n+1})$ is a weak
solution of Equation (\ref{++1.2}). Then $u\in
G^s(\mathbb{R}^{2n+1})$ for any $s\geq \frac{2}{\delta}$, provided
that the coefficient $ a\in G^s(\mathbb{R}^{2n+1})$, $a(t,x,v)>0$
and the nonlinear function $ F(t, x, v, q)\in
G^s(\mathbb{R}^{2n+2}).$
\end{thm}

\begin{rmk}
Our results here are  local  interior regularity results. This
implies that if there exists a weak solution in $\D'$, then the
solution is in Gevrey class in the interior of the domain. Thus, the
interior regularity of a weak solution does not depend much on the
regularity of the initial Cauchy data. Also, without loss of
generality, we can assume that $a(t,x,v)\geq c_0>0$ for all $(t, x,
v)\in\RR^{2n+1}$ with $c_0$ a positive constant, and all derivatives
of the coefficient $a$ are bounded in $\RR^{2n+1}$.
\end{rmk}

The paper is organized as follows: in section \ref{sect2}, we prove
that $\mc P$ is  subelliptic by using the method of  subelliptic
multiplier developed by J. Kohn [\ref{Kohn}]. Section \ref{sect+3}
is devoted to the study of  the commutator of
$(-\widetilde\triangle_v)^\sigma$ with the cut-off function in the
$v$ variable. In section \ref{sect3+}, we use the subelliptic
estimates to prove the Gevrey hypoellipticity of the operator $\mc
P$. Section \ref{sect4} is devoted to the proof of the Gevrey
regularity for the weak solution of the semilinear kinetic equation
(1.3).


\section{Subelliptic estimates}
\label{sect2} \setcounter{equation}{0}

In this paper, the notation, $\|\cdot\|_\kappa, \kappa\in\mathbb R,$
is used for the classical Sobolev norm in $H^\kappa(\mathbb
R^{2n+1})$, and
 $(f,~g)$ is the inner product of $f, g\in L^2(\bb R^{2n+1})$. Moreover
 if $f,~g\in C_0^\infty(\bb R^{2n+1}),$ it is easy to see that
\begin{eqnarray}\label{Young}
|(f,~g)|\leq \|f\|_\kappa \|g\|_{-\kappa}\leq
\frac{\varepsilon\|f\|_\kappa^2}{2}+\frac{\|g\|_{-\kappa}^2}{2\varepsilon}.
\end{eqnarray}

We have also the interpolation inequality in Sobolev space: For any
$\varepsilon>0$ and $r_1< r_2<r_3,$
\begin{eqnarray}\label{interpolation}
\|f\|_{r_2}\leq \varepsilon
\|f\|_{r_3}+\varepsilon^{-(r_2-r_1)/(r_3-r_2)}\|f\|_{r_1}.
\end{eqnarray}

Let $\Omega$ be an open subset of $\bb R^{2n+1}$ and $S^m(\Omega),
m\in \bb R,$ be the symbol space of the classical
pseudo-differential operators (when there is no risk to cause the
confusion, we will simply write $S^m$ for $S^m(\Omega)$). We say
$P=P(t,x,v,D_t,D_x,D_v)\in{\rm Op}(S^m)$ to be a pseudo-differential
operator of order $m,$ if its symbol $p(t,x,v;\tau,\xi,\eta)\in
S^m.$ If $P\in{\rm Op}(S^m)$, then $P$ is a continuous operator from
$H_{c}^\kappa(\Omega)$ to $H_{loc}^{\kappa-m}(\Omega)$, where
$H_{c}^\kappa(\Omega)$ is the subspace of $H^\kappa(R^{2n+1})$ which
consists of the distributions having their compact support in
$\Omega$. $H_{loc}^{\kappa-m}(\Omega)$ consists of the distributions
$h$ such that $\phi h\in H^{\kappa-m}(\bb R^{2n+1})$ for any
$\phi\in C_0^\infty(\Omega)$. For more details on the
pseudo-differential operators, we refer to Treves [\ref{Treves}].
Observe that if $P_1\in {\rm Op}(S^{m_1})$, $P_2\in {\rm
Op}(S^{m_2})$, then $[P_1,~P_2]\in {\rm Op}(S^{m_1+m_2-1}).$

\bigbreak
 We study now the operator $\mc P$ given by
(\ref{Fokker-Planck}). For simplicity, we introduce the following
notations
$$
{\widetilde\Lambda}^{\sigma}_v=(-\widetilde\triangle_v)^{\frac{\sigma}{2}},
\hspace{0.3cm}X_0=\partial_t+v\cdot\partial_x ,
\hspace{0.3cm}X_j=\partial_{v_j},  j=1, \cdots, n ,
$$
$$
\Lambda^\kappa=(1+|D_{t}|^2+|D_{x}|^2+|D_{v}|^2)^{\kappa/2}.
$$
Then $\mc P$ can be written as
$\mc{P}=X_0+a(t,x,v){\widetilde\Lambda}^{2\sigma}_v,$ and
$\partial_{x_j}=[X_j,~X_0].$ The following simple fact is used
frequently: For any compact $K\subset {\mathbb R^{2n+1}} $ and
$r\geq 0$, there exists $C_{K, r}>0$ such that for any $f\in
C_0^\infty(K),$
\begin{eqnarray}\label{T}
\|{\widetilde\Lambda}^{\sigma}_v f\|_r \leq C_{K, r}\bigset{\|\mc
Pf\|_r+\|f\|_r}.
\end{eqnarray}
In fact, a simple computation gives that
\begin{eqnarray*}
\|{\widetilde\Lambda}^{\sigma}_v f\|_r^2 &=&\RE(\mc Pf,
~a^{-1}\Lambda^{2r}f)-
\RE(X_0f, ~a^{-1}\Lambda^{2r}f)\\
&=&\RE(\mc Pf, ~a^{-1}\Lambda^{2r}f)- {\frac 1 2}(f,~
[a^{-1}\Lambda^{2r},~\widetilde{X}_0]f)-
{\frac 1 2}(f, ~[\Lambda^{2r},~a^{-1}]~\widetilde{X}_0f)\\
&\leq&C_{K, r}\bigset{\|\mc Pf\|_r+\|f\|_r},
\end{eqnarray*}
where $\widetilde{X}_0=\partial_t+\tilde\psi(v) v\cdot\partial_x$
and $\tilde\psi \in C^\infty_0(\RR^n_v)$ is a cutoff function in the
$v$ variable such that $\tilde\psi=1$ in the projection of $K$ on
$\RR^n_v$. Remark that, with the choice of  such a cutoff function,
we have that
$$
\widetilde{X}_0 P(t, x, v, D_t, D_x, D_v) f={X}_0 P(t, x, v, D_t,
D_x, D_v) f
$$
for any $f\in C_0^\infty(K)$ and any partial differential operator
$P(t, x, v, D_t, D_x, D_v)$.

\bigskip

First we show  $\mc P$ is a subelliptic operator on $\RR^{2n+1}$
with a gain of order
$\delta=\max\set{\frac{\sigma}{4},~\frac{\sigma}{2}-\frac 1 6}.$

\begin{prp}\label{prp2}

Let $K$ be a compact subset of $\bb{R}^{2n+1}.$  For any $r\geq 0,$
there exists a constant $C_{K,r},$ depending only on $K$ and $r,$
such that for any $f\in C_0^\infty(K),$
\begin{eqnarray}
\label{sub0} \|f\|_{r+\delta}\leq
C_{K,r}\{~\|\mc{P}f\|_r+\|f\|_0~\},
\end{eqnarray}
where $\delta=\max\set{\frac{\sigma}{4},~\frac{\sigma}{2}-\frac 1
6}.$

\end{prp}

In order to prove Proposition \ref{prp2}, we need the following two
lemmas.

\begin{lem}\label{lem2}
Let $K$ be any compact subset of $\bb{R}^{2n+1}.$ Then for any $f\in
C_0^\infty(K),$ we have
\begin{equation}
\label{re2}
\begin{array}{l}
 \|\Lambda^{-1/3}X_0f\|_0 \leq
 C_K(~\|\mc{P}f\|_0+\|f\|_0~),\hspace{2.2cm}
\end{array}
\end{equation}
and
\begin{equation}
\label{re1}
\begin{array}{l}
\|\Lambda^{-1}X_jf\|_\sigma \leq
C_K(~\|\mc{P}f\|_0+\|f\|_0~),\hspace{0.5cm} j=1,\cdots ,n.
\end{array}
\end{equation}
\end{lem}

This is the result of Proposition 3.1 in \cite{MX}. The following
lemma is to estimate the commutators, which is different from the
calculation in \cite{MX} for the second part of the lemma.

\begin{lem}\label{mul}

Let $K$ be a compact subset of $\bb{R}^{2n+1}.$ Then for any $f\in
C_0^\infty(K),$ we have
\begin{equation}
\label{mul1}
\begin{array}{l}
\|[X_j,~\Lambda^{-1}{\widetilde{X}}_0]f\|_{\sigma/2-1/6} \leq
C_K(~\|\mc{P}f\|_0+\|f\|_0~),\hspace{0.5cm} j=1,\cdots ,n,
\end{array}
\end{equation}
and
\begin{equation}
\label{mul2}
\begin{array}{l}
 \|[\Lambda^{-1}X_j,~{\widetilde{X}}_0]f\|_{\sigma/4} \leq
 C_K(~\|\mc{P}f\|_0+\|f\|_0~),\hspace{0.5cm} j=1,\cdots ,n.
\end{array}
\end{equation}

\end{lem}

\begin{proof} We denote
$Q_j=\Lambda^{\sigma-1/3-1 }[X_j, {X}_0]=
\Lambda^{\sigma-1/3-1}\partial_{x_j}\in {\rm Op}(S^{\sigma-1/3}).$
Note that $[X_k,~Q_j]=0$ for any $1\leq k\leq n$. Therefore for any
$f\in C_0^\infty(K),$
\begin{align*}
&\|[X_j,~\Lambda^{-1}\widetilde{X}_0]f\|_{\sigma/2-1/6}^2=
\|[X_j,~\Lambda^{-1}X_0]f\|_{\sigma/2-1/6}^2\\
&\leq |(X_j\Lambda^{-1}X_0f,~Q_jf)|+|(\Lambda^{-1}{\widetilde{X}}_0
 X_jf,~Q_jf)|\\
&\leq|(\Lambda^{-1}X_0f,~Q_jX_jf)|+|(X_jf,~{\widetilde{X}}_0\Lambda^{-1}Q_jf)|\\
&\leq~\|\Lambda^{-1}X_0f\|_{2/3}\|Q_j X_jf\|_{-2/3} +|(X_jf,
~[{\widetilde{X}}_0,~\Lambda^{-1}Q_j]f)|
+|(X_jf, ~\Lambda^{-1}Q_jX_0f)|~\\
&\leq
C_K\{~\|\Lambda^{-1/3}X_0f\|_{0}^2+\|\Lambda^{-1}X_jf\|_\sigma^2+\|f\|_0^2
~\},
\end{align*}
where we have used the simple fact that
$[{\widetilde{X}}_0,~\Lambda^{-1}Q_j]\in {\rm
Op}(S^{\sigma-1/3-1}).$ Then (\ref{re2}) and (\ref{re1}) give
immediately (\ref{mul1}).

We now study (\ref{mul2}). First of all, we have
\begin{eqnarray*}
\|[\Lambda^{-1}X_j,~{\widetilde{X}}_0]f\|_{\sigma/4}^2&=&(\Lambda^{-1}X_j
{\widetilde{X}}_0f,~\Lambda^{\sigma/2}[\Lambda^{-1}X_j,~{\widetilde{X}}_0]f)\\
&&- ({\widetilde{X}}_0 \Lambda^{-1}X_jf,
~\Lambda^{\sigma/2}[\Lambda^{-1}X_j, ~{\widetilde{X}}_0]f).
\end{eqnarray*}
By a straightforward calculation, it follows that
\begin{align*}
&|({\widetilde{X}}_0 \Lambda^{-1}X_jf,
~\Lambda^{\sigma/2}[\Lambda^{-1}X_j,
~{\widetilde{X}}_0]f)|=|(\Lambda^{-1}X_j f,
~{\widetilde{X}}_0\Lambda^{\sigma/2}
[\Lambda^{-1}X_j, ~{\widetilde{X}}_0]f)|\\
&\leq|(\Lambda^{-1}X_j f, ~\Lambda^{\sigma/2}[\Lambda^{-1}X_j,
~{\widetilde{X}}_0] {\widetilde{X}}_0f)| +|(\Lambda^{-1}X_j
f,~[\Lambda^{\sigma/2}[\Lambda^{-1}X_j,
~{\widetilde{X}}_0],~{\widetilde{X}}_0]f)|\\
&\leq C_K\big\{~|(\Lambda^{-1}X_j f, ~
\Lambda^{\sigma/2}[\Lambda^{-1}X_j,
~{\widetilde{X}}_0]X_0f)|+\|\Lambda^{-1}X_j
f\|_{\sigma/2}^2+\|f\|_0^2~\big\}\\
&\leq C_K\big\{~|(\Lambda^{-1}X_j f,
~\Lambda^{\sigma/2}[\Lambda^{-1}X_j,
~{\widetilde{X}}_0]X_0f)|+\|\mc{P}f\|_0^2+\|f\|_0^2~\big\}.
\end{align*}
In the last inequality,  we have used (\ref{re1}) in Lemma
\ref{lem2}.

Denote $P_{\sigma/2}=\Lambda^{\sigma/2}[\Lambda^{-1}X_j,
~{\widetilde{X}}_0] \in {\rm Op}(S^{\sigma/2})$. Recall that
$X_0=\mc P-a{\widetilde\Lambda}^{2\sigma}_v$.  We have
\begin{align*}
&|(\Lambda^{-1}X_j f, ~\Lambda^{\sigma/2}[\Lambda^{-1}X_j,
~{\widetilde{X}}_0]X_0f)|
=|(\Lambda^{-1}X_j f,~P_{\sigma/2}X_0f)|\\
&\leq |(\Lambda^{-1}X_j f,~P_{\sigma/2}\mc{P}f)| +|(\Lambda^{-1}X_j
f,~P_{\sigma/2}a{\widetilde\Lambda}^{2\sigma}_v
 f)|\\
&\leq C_K\{~\|\Lambda^{-1}X_j f\|_{\sigma/2}^2+\|\mc{P}f\|_0^2
+|({\widetilde\Lambda}^{\sigma}_v\Lambda^{-1}X_j f,
~{\widetilde\Lambda}^{-\sigma}_v
P_{\sigma/2}a{\widetilde\Lambda}^{2\sigma}_v f)|~\}\\
&\leq C_K\{~\|\Lambda^{-1}X_j f\|_{\sigma/2}^2+\|\mc{P}f\|_0^2
+\|{\widetilde\Lambda}^{\sigma}_v\Lambda^{-1}X_j f\|_{\sigma/2}^2+
\|{\widetilde\Lambda}^{\sigma}_vf\|_0^2~\}\\
&\leq
C_K\{~\|{\widetilde\Lambda}^{\sigma}_v\Lambda^{\sigma/2}\Lambda^{-1}X_j
f\|_0^2+\|\mc{P}f\|_0^2+\|f\|_0^2~\}.
\end{align*}
For the last inequality,  we used results from  (\ref {T}) and
(\ref{re1}). Clearly,
$[{\widetilde\Lambda}^{\sigma}_v,~\Lambda^{-1}X_j]=[{\widetilde\Lambda}^{\sigma}_v,
~\Lambda^{\sigma/2}]=[\Lambda^{-1}X_j,~\Lambda^{\sigma/2}]=0$. Then
we get
\begin{align*}
&\|{\widetilde\Lambda}^{\sigma}_v\Lambda^{\sigma/2}\Lambda^{-1}X_jf\|_0^2=
-{\rm Re}(\mc{P}f, ~a^{-1}\Lambda^\sigma \Lambda^{-2}X_j^2f) +{\rm
Re}({\widetilde{X}}_0f, ~a^{-1}\Lambda^\sigma
 \Lambda^{-2}X_j^2f)\\
&\leq C_K\{~\|\mc Pf\|_0^2+\|\Lambda^{-1}X_jf\|_\sigma^2 +{\frac 1
2}|(f,~[\Lambda^\sigma
\Lambda^{-2}X_j^2,~a^{-1}{\widetilde{X}}_0]f)|\\
&\,\,\,\,\,\,\,\,\,\,\,+ {\frac 1
2}|(f,~[a^{-1},~{\widetilde{X}}_0]\Lambda^\sigma
 \Lambda^{-2}X_j^2f)|~\}\\
&\leq C_K\{~\|\mc Pf\|_0^2+\|f\|_0^2+\|\Lambda^{-1}X_jf\|_\sigma^2
+|(f,~\Lambda^{-1}X_j[\Lambda^\sigma
\Lambda^{-1}X_j,~a^{-1}{\widetilde{X}}_0]f)|\\
&\,\,\,\,\,\,\,\,\,\,\,+
|(f,~[\Lambda^{-1}X_j,~a^{-1}{\widetilde{X}}_0]\Lambda^\sigma
 \Lambda^{-1}X_j f)|~\}\\
&\leq C_K\{~\|\mc Pf\|_0^2+\|\Lambda^{-1}X_jf\|_\sigma^2+\|f\|_0^2~\}\\
&\leq C_K\{~\|\mc Pf\|_0^2+\|f\|_0^2~\}.
\end{align*}
The above three estimates show immediately
\begin{eqnarray*}
|({\widetilde{X}}_0\Lambda^{-1}X_jf,~P_{\sigma/2}f)| \leq
C_K\{~\|\mc P f\|_0^2+\|f\|_0^2~\}.
\end{eqnarray*}
Similarly, we can prove
\begin{eqnarray*}
|(\Lambda^{-1}X_j {\widetilde{X}}_0 f,~P_{\sigma/2}f)| \leq
C_K\{~\|\mc P f\|_0^2+\|f\|_0^2~\}.
\end{eqnarray*}
This completes the proof of  Lemma \ref{mul}.
\end{proof}

The rest of this section is devoted to the proof of proposition
\ref{prp2}:

\smallskip

\emph{Proof of Proposition \ref{prp2}.}\hspace{0.1cm} Notice that
$\partial_{x_j}=[X_j,~X_0]$ and $
\partial_t=X_0-\sum\limits_{j=1}^n
v_j\cdot[X_j,~X_0].$ Hence, for any $f\in C_0^\infty(K),$ we have
\begin{eqnarray*}
\|f\|^2_{\delta}
&=&~\|\partial_tf\|^2_{\delta-1}+\sum\limits_{j=1}^n\|\partial_{x_j}f\|^2_{\delta-1}
+\sum\limits_{j=1}^n\|\partial_{v_j}f\|^2_{\delta-1}+\|f\|^2_0~\\
&\leq&C_K~\{~\|\Lambda^{-1}X_0f\|^2_{\delta}+
\sum\limits_{j=1}^n\big(\|\tilde\psi(v) v_j
 [X_j,~{\widetilde{X}}_0]f\|^2_{\delta-1}\\
&&\,\,\,\,\,\,\,\,\,\, +
\|[X_j,~{\widetilde{X}}_0]f\|^2_{\delta-1}+\|\Lambda^{-1}X_jf\|^2_{\delta}\big)+\|f\|^2_0~\}.
\end{eqnarray*}
Since
$\delta=\max\set{\sigma/4,~\sigma/2-1/6}\leq\min\set{2/3,~\sigma},$
applying (\ref{re2}) and (\ref{re1}) to Lemma \ref{lem2}, we have
that
$$
\|\Lambda^{-1}X_0f\|_{\delta}+\sum_{j=1}^n\|\Lambda^{-1}X_jf\|_{\delta}\leq
C_K\{~\|\mc{P}f\|_0+\|f\|_0~\}
$$
and
$$\|\tilde\psi(v) v_j [X_j,~{\widetilde{X}}_0]f\|_{\delta-1}\leq C_K\{
 \|
[X_j,~{\widetilde{X}}_0]f\|_{\delta-1}+\|f\|_{0}\}.
$$
 It remains to treat the  term $\|[X_j,
\,{\widetilde{X}}_0]f\|_{\delta-1}.$ We consider the following two
cases.

\smallskip

\emph{Case (i).}
 \:$\delta=\max\set{\sigma/4,~\sigma/2-1/6}=\sigma/2-1/6.$

We apply (\ref{mul1}) in Lemma \ref{mul} to get
\begin{eqnarray*}
\|[X_j,~{\widetilde{X}}_0]f\|_{\delta-1}&\leq& \|[X_j,
~\Lambda^{-1}{\widetilde{X}}_0]f\|_\delta+\|[X_j,~\Lambda^{-1}]
{\widetilde{X}}_0 f\|_\delta\\
&\leq&C_K\{~\|\mc{P}f\|_0+\|\Lambda^{-1}X_0f\|_\delta+ \|f\|^2_0~\}.
\end{eqnarray*}
Since $\delta<2/3,$ then applying (\ref{re2}) again, we get
immediately
\begin{eqnarray*}
\|[X_j,~{\widetilde{X}}_0]f\|_{\delta-1} \leq
C_K\{~\|\mc{P}f\|_0+\|f\|_0~\}.
\end{eqnarray*}

\smallskip

\emph{Case (ii).} \:$\delta=\max(\sigma/4, \sigma/2-1/6)=\sigma/4.$

By (\ref{mul2}) in Lemma \ref{mul}, it follows that
\begin{eqnarray*}
\|[X_j,~{\widetilde{X}}_0]f\|_{\delta-1}&\leq& \|[\Lambda^{-1}X_j,
~{\widetilde{X}}_0]f\|_\delta+\|
[\Lambda^{-1}, ~{\widetilde{X}}_0]X_jf\|_\delta\\
&\leq&C_K\{~\|\mc{P}f\|_0+\|\Lambda^{-1}X_jf\|_\delta+\norm{f}_0~\}.
\end{eqnarray*}
Note that $\delta<\sigma,$ and hence from (\ref{re1}), we have
\begin{eqnarray*}
\|[X_j,~{\widetilde{X}}_0]f\|_{\delta-1}\leq
C_K\{~\|\mc{P}f\|_0+\|f\|_0~\}.
\end{eqnarray*}

\smallskip

A combination of  Case (i) and  Case (ii) yields that for
$\delta=\max\set{\sigma/4,~\sigma/2-1/6},$
\begin{eqnarray*}
\|[X_j,~{\widetilde{X}}_0]f\|_{\delta-1} \leq
C_K\{~\|\mc{P}f\|_0+\|f\|_0~\}.
\end{eqnarray*}
Then we get
\begin{eqnarray}\label{delta}
\|f\|_\delta \leq C_K\{~\|\mc{P}f\|_0+\|f\|_0~\}.
\end{eqnarray}

\bigbreak

Choose now a cutoff function $\psi\in C_0^{\infty}(\bb R^{2n+1})$
such that $\psi|_K\equiv1$ and Supp $\psi$ is a neighborhood of $K$.
Then for any $r\geq0$,  $\varepsilon>0$ and  $f\in C^\infty_0(K)$,
by (\ref{delta}), we have
\begin{eqnarray*}
\|f\|_{r+\delta}&=&\|\Lambda^r \psi
f\|_{\delta}\leq\|\psi\Lambda^rf\|_\delta+\|[\Lambda^r,~\psi]f\|_\delta
\leq C_K\{~\|\mc{P}\psi\Lambda^rf\|_0+\|f\|_r~\}.
\end{eqnarray*}
Furthermore, notice  that
$$
[a {\widetilde\Lambda}^{2\sigma}_v,
\psi\Lambda^r]=2a[{\widetilde\Lambda}^{\sigma}_v,
~\psi\Lambda^r]{\widetilde\Lambda}^{\sigma}_v
+a[{\widetilde\Lambda}^{\sigma}_v,~
[{\widetilde\Lambda}^{\sigma}_v,~\psi\Lambda^r]~]+[a,
~\psi\Lambda^r]{\widetilde\Lambda}^{2\sigma}_v.
$$
Hence
\begin{eqnarray*}
\|\mc{P}\psi\Lambda^rf\|_0&\leq&
\|\psi\Lambda^r\mc{P}f\|_0+\|[{\widetilde{X}}_0,~\psi\Lambda^r]f\|_0+
\|a[{\widetilde\Lambda}^{\sigma}_v,~[{\widetilde\Lambda}^{\sigma}_v,\psi\Lambda^r]~]f\|_0\\
&&+2\|a[{\widetilde\Lambda}^{\sigma}_v,~\psi\Lambda^r]{\widetilde\Lambda}^{\sigma}_v
f\|_0 +\|[a,
~\psi\Lambda^r]{\widetilde\Lambda}^{2\sigma}_v f\|_0\\
&\leq&C_{K,
r}\{~\|\mc{P}f\|_r+\|f\|_r+\|{\widetilde\Lambda}^{\sigma}_v
f\|_r~\},
\end{eqnarray*}
Combining  with (\ref{T}), we have
\begin{eqnarray*}
\|\mc{P}\psi\Lambda^rf\|_0 \leq C_{K, r}\{~\|\mc{P}f\|_r+\|f\|_r~\}.
\end{eqnarray*}
The above three estimates show that
\begin{eqnarray*}
\|f\|_{r+\delta}\leq C_{K, r}\{~\|\mc Pf\|_r+\|f\|_r~\}.
\end{eqnarray*}
Applying the interpolation inequality (\ref{interpolation}), it
follows that
\[
\|f\|_{r+\delta}\leq C_{\varepsilon,r,
K}\{~\|\mc{P}f\|_r+\|f\|_0~\}+\varepsilon\|f\|_{r+\delta}.
\]
Taking $\varepsilon$ small enough, we get the desired subelliptic
estimate (\ref {sub0}). This completes the proof of Proposition
\ref{prp2}.

\bigbreak Since the subelliptic estimate in  Proposition \ref{prp2}
is true for $0<\sigma<1$, we can now improve the
$C^\infty$-hypoellipticity result of [\ref {MX}]( which is for
$1/3<\sigma<1$ ) as in the  following Theorem:

\begin{thm}\label{th2.1}
Let $0<\sigma<1$. Then the operator $\mc{P}$ given by
(\ref{Fokker-Planck}) is $C^\infty$ hypoelliptic in
$\mathbb{R}^{2n+1}$, provided that the coefficient $ a(t,x,v) $ is
in the space  $C^\infty(\mathbb{R}^{2n+1})$ and $a(t,x,v)>0$ .
\end{thm}
In fact, if we consider only the local regularity problem, as in
Proposition 4.1 of [\ref {MX}], we can prove that if
 $f\in H^s_{{loc}}(\RR^{2n+1}), u\in\cD'(\RR^{2n+1})$ and $\mc{P}u=f$
then $u\in H^{s+\delta}_{{loc}}(\RR^{2n+1})$. By using the
subelliptic estimate (\ref {sub0}), the estimate for the commutators
between the operator $\mc{P}$ and the mollifiers  are exactly the
same as in  Section 4 of [\ref {MX}]. This gives the $C^\infty$
hypoellipticity by the  Sobolev embedding theorem. The same argument
applies to the semi-linear equations.

Remark that the results of [\ref {MX}] are not only  regularity
results. The authors also proved a global estimate with weights (the
moments). This  is another important problem for the kinetic
equation.


\section{Cutoff functions and commutators}
\label{sect+3} \setcounter{equation}{0}

To prove the Gevrey regularity of a solution, we have to prove an
uniformly iteration estimate (\ref{gevrey}). Our only tool is the
subelliptic estimate (\ref {sub0}).  Since it is a local estimate,
we have to control the commutators between the operator $\mc{P}$ and
the cutoff functions. This is always the technical key step in the
Gevrey regularity problem. Our additional difficulty comes from the
complicated nature of the operator $\mc{P}$.

Since the Gevrey hypoellipticity is a local property, it suffices to
show $\mc{P}$ is Gevrey hypoelliptic in the open domain
$\Omega\subset\Real^{2n+1}$ given by
\[
\Omega=\Omega^1\times \Omega^2=\{(t,x)\in \bb{R}^{n+1};\,
t^2+|x|^2<1 \}\times \set{v\in \bb{R}^{n};\,  |v|^2<1}.
\]
Define $W$ by setting
 \[ W=2\Omega=\set{\inner{t,x,v};\: \abs
t^2+\abs x^2\leq 2^2, \:\abs v\leq 2}
\]
For $0\leq \rho<1$, set
$\Omega_\rho=\Omega_\rho^1\times\Omega_\rho^2$ with $\Omega_\rho^1$
and $\Omega_\rho^2$ to be given by
  \[
  \Omega_\rho^1=\set{(t,x)\in\bb{R}^{n+1}; \:\:
  \biginner{t^2+|x|^2}^{1/2}<1-\rho},\qquad
   \Omega_\rho^2=
   \set{v\in \bb{R}^{n};\:\: |v|^2<1-\rho}.
\]
Let $\chi_{\rho}$  be the characteristic function of the set
$\Omega_{\rho}^2$, and let $\phi\in C_0^\infty(\Omega^2)$ be a
function satisfying  $0\leq \phi\leq 1$ and
$\int_{\bb{R}^{n}}\phi(v)dv=1.$ For any $\eps, ~\tilde\eps>0$,
setting $\phi_\varepsilon(v)=\varepsilon^{-n}
\phi\inner{{\frac{v}{\varepsilon}}}$ and
$\varphi_{\varepsilon,\tilde\eps}(v)=
\phi_{\varepsilon/2}*\chi_{\varepsilon/2 +\tilde\eps}(v)$. Then for
a small $\eps, ~\tilde\eps>0$,
\begin{eqnarray*}
&&\varphi_{\varepsilon,\tilde\eps}\in
C_0^\infty(\Omega_{\tilde\eps}^2);\,\,\,\,\,\,\varphi_{\eps,
\tilde\eps}=1\,\,\,\, \mbox{in}\,\,\,
\Omega_{\varepsilon+\tilde\eps}^2;\\
&& \sup_{v\in\RR^n}\abs{D^\alpha\varphi_{\varepsilon,\tilde\eps}(v)}
\leq C_\alpha\varepsilon^{-|\alpha|} \,\,\,\,\,\,\,\, \mbox{for
any}\,\,\,\,\,\,\, \alpha\in\NN^n.
\end{eqnarray*}
In the same way, we can find a function
$\psi_{\eps,\tilde\eps}(t,x)\in C_0^\infty(\Omega_{\tilde\eps}^1)$
such that $\psi_{\eps, \tilde\eps}=1$ in
$\Omega_{\varepsilon+\tilde\eps}^1$ and
$\sup\abs{D^\alpha\psi_{\varepsilon,\tilde\eps}} \leq
C_\alpha\varepsilon^{-|\alpha|}.$

Now for any $N\in\NN, N\geq 2$ and any $0<\rho<1$,  we set
$$
\Phi_{\rho, N}(t,x,v)=\psi_{\frac{\rho}{N}, \frac{(N-1)\rho}
{N}}(t,x) \vpi_{\frac{\rho}{N}, \frac{(N-1)\rho} {N}}(v).
$$
Then we have,
\begin{equation}\label{cutoff}
   \left\{ \begin{array}{lll}
            \Phi_{\rho, N}\in C_0^\infty(\Omega_{\frac{N-1}
            {N}\rho})\\
            \smallskip
            \Phi_{\rho,N}(t,x,v)=1, \quad (t,x,v)\in\Omega_{\rho},
            \\
            \smallskip
            \sup\abs{D^\alpha\Phi_{\rho,N}}
            \leq C_\alpha(N/\rho)^{|\alpha|}.
            \end{array}
   \right.
\end{equation}

\vspace{0.5ex}
 For such cut-off functions, we have the following Lemma (see  Corollary 0.2.2 of
 \cite{Dur}).
\begin{lem}\label{3.1}
  There exists a constant $C_n,$ depending only on $n,$ such that
  for any $0\leq\mu\leq n+2,$ and  $f\in\cS(\Real^{n+1}),$ we have
  \begin{equation}\label{cutoffnorm}
   \norm{\inner{D^\gamma\Phi_{\rho,N}}f}_{\mu}\leq C_n\set{(N/\rho)^{\abs\gamma}\norm
   f_{\mu}+(N/\rho)^{\abs\gamma+\mu}\norm{f}_0}, \quad \abs\gamma\leq 2.
  \end{equation}
\end{lem}

\vspace{1ex}

We study now the commutator of above cutoff function with the
operator $\mc{P}$. Since the operator is a differential operator
with respect to the $(t, x)$ variables, it is enough to consider the
commutator of ${\widetilde\Lambda}^{\sigma}_v$ with a cut-off
function in the $v$ variable. We set $\varphi_{\rho,N}(v)
=\vpi_{{\rho\over N},{{(N-1)\rho}\over N}}(v) $. The proof of the
following Lemma is very similar to that of M. Durand \cite{Dur}.
Since  our calculus  is much  easier and much more direct, we repeat
it here.

\begin{lem}\label{+lem9}
There exists a constant $C_{\sigma,n}$, depending only on $n$ and
$\sigma$, such that for any $\kappa$ with $1\leq \kappa\leq n+3,$
and  $f\in\cS(\Real^{2n+1}),$
\begin{equation}\label{+com1}
  \|[{\widetilde\Lambda}^{\sigma}_v,~~\varphi_{\rho,N}] f \|_\kappa
    \leq C_{\sigma,n}\set{
\inner{N/\rho}^\sigma\norm{f}_{\kappa}
+\inner{N/\rho}^{\kappa+\sigma}\norm{f}_0}
\end{equation}
and
\begin{eqnarray}\label{+com2}
\|[{\widetilde\Lambda}^{\sigma}_v,~~[{\widetilde\Lambda}^{\sigma}_v,~~\varphi_{\rho,N}]~]
f \|_\kappa\leq  C_{\sigma,n}\set{
\inner{N/\rho}^{2\sigma}\norm{f}_{\kappa}
+\inner{N/\rho}^{\kappa+2\sigma}\norm{f}_0}.
\end{eqnarray}

\end{lem}

\begin{rmk}
Observe for $\tilde\rho=\frac{(N-1)\rho}{N},$\:
$\varphi_{\rho,N}{\widetilde\Lambda}^{\sigma}_v(1-\varphi_{\tilde\rho,N})f
=-\varphi_{\rho,N}\,[{\widetilde\Lambda}^{\sigma}_v,
\varphi_{\tilde\rho,N}]\,f.$ Then as a consequence of \reff{+com1},
we have
  \begin{equation*}
  \|\varphi_{\rho,N}{\widetilde\Lambda}^{\sigma}_v(1-\varphi_{\tilde\rho,N})f \|_\kappa
    \leq C_{\sigma,n}\set{
\inner{N/\rho}^\sigma\norm{f}_{\kappa}
+\inner{N/\rho}^{\kappa+\sigma}\norm{f}_0}.
\end{equation*}
Hence, in the following, we omit the detailed discussions for such
terms.
\end{rmk}

\begin{proof} To simplify the notation, in the course of the proof, we shall use $C$
 to
denote a constant which depend only on $n$ and $\sigma$ and may be
different in different contexts. We denote by $(\tau, \xi, \eta)$
the Fourier transformation variable of $(t,x,v)$. $\mc F_{t,x}(g),
\:\mc F_v(g)$ are the partial Fourier transforms, and $\hat g$ is
the full Fourier transform with respect to $(t,x,v).$ Set
\[
h=[{\widetilde\Lambda}^{\sigma}_v,~\vpi_{\rho,N}]f, \quad
H(v)=H_{\tau,\xi}(v)=\mc F_{t,x}(f)(\tau, \xi, v)
\]
In the following discussion, we always write $H(v)$ for
$H_{\tau,\xi}(v)$, if there is  no risk of causing the  confusion.
It is clear that
\begin{equation}\label{relation}
\mc F_{t,x}(h)(\tau, \xi,
v)=[{\widetilde\Lambda}^{\sigma}_v,~\vpi_{\rho,N}]\mc
F_{t,x}(f)(\tau,\xi,v)=[{\widetilde\Lambda}^{\sigma}_v,~\vpi_{\rho,N}]
H(v) .
\end{equation}
Observe that the desired inequality \reff{+com1} will follow if we
show that, for each fixed pair $(\tau,\xi),$
\begin{equation}\label{+++com1}
  \norm{\bigcom{\widetilde\Lambda_v^\sigma,\:\varphi_{\rho,N}(\cdot)}
  H(\cdot)}_{H^\kappa(\Real^n_v)}
  \leq C\Big\{\inner{N/\rho}^\sigma\norm{
  H(\cdot)}_{H^\kappa(\Real^n_v)}
  +\inner{N/\rho}^{\kappa+\sigma}
  \norm{H(\cdot)}_{L^2(\Real^n_v)}\Big\}.
\end{equation}
Indeed, a direct computation yields that
\begin{align*}
&\norm{h}_\kappa^2=\int_{\bb
R^{2n+1}}(1+\tau^2+|\xi|^2+|\eta|^2)^\kappa \big|\hat{
h}(\tau,\xi,\eta)\big|^2d\tau d\xi d\eta\\
&\leq C\int_{\bb
R^{2n+1}}\set{(1+\tau^2+|\xi|^2)^\kappa+|\eta|^{2\kappa}} \big|\hat{
h}(\tau,\xi,\eta)\big|^2d\tau d\xi d\eta\\
&\leq C\int_{\bb R^{n+1}}(1+\tau^2+|\xi|^2)^\kappa\inner{\int_{\bb
R^{n}}\inner{1+|\eta|^2}^{\kappa}\big|\hat{ h}(\tau,\xi,\eta)\big|^2
d\eta}d\tau d\xi\\
&=C\int_{\bb
R^{n+1}}(1+\tau^2+|\xi|^2)^\kappa\inner{\norm{\bigcom{\widetilde\Lambda_v^\sigma,
\:\varphi_{\rho,N}(\cdot)}
  H_{\tau,\xi}(\cdot)}_{H^\kappa(\Real^n_v)}}d\tau d\xi.
\end{align*}
This along with \reff{+++com1} yields the desired inequality
\reff{+com1}.

Next, we shall prove \reff{+++com1}.  First, for any
$g\in\cS(\Real^{n}),$ we have
\begin{equation}\label{FourierM}
  |D_v|^\sigma\, g(v)=C_\sigma\int_{\Real^n}\frac{g(v)-g(v-\tilde
 v)}{\abs
  {\tilde v}^{n+\sigma}}\,d\tilde v
\end{equation}
with $C_\sigma\neq 0$ being a complex constant depending only on
$\sigma$ and the dimension $n.$

In fact,
\[
  \int_{\Real^n}\frac{g(v)-g(v-\tilde v)}{\abs
  {\tilde v}^{n+\sigma}}\,d\tilde v=\int_{\Real^n}\mc
  F_v(g)(\eta)\,e^{i\,v\cdot\eta}\left(
  \int_{\Real^n}\frac{1-e^{-i\,\tilde v\cdot\eta}}{\abs
  {\tilde v}^{n+\sigma}}\,d\tilde v\right)\,d\eta
\]
On the other hand, it is clear that
\[
  \int_{\Real^n}\frac{1-e^{-i\,\tilde v\cdot\eta}}{\abs
  {\tilde v}^{n+\sigma}}\,d\tilde v=\abs\eta^\sigma \int_{\Real^n}
  \frac{1-e^{-i\, u\cdot\frac{\eta}{\abs\eta}}}{\abs
  {u}^{n+\sigma}}\,d u.
\]
Observe that $\int_{\Real^n}\frac{1-e^{i\,
u\cdot\frac{\eta}{\abs\eta}}}{\abs
  {u}^{n+\sigma}}\,d u \neq 0$ is a complex constant depending only on
 $\sigma$ and
the dimension $n$, but independent of $\eta.$  Then the above two
equalities give \reff{FourierM}.

Next, we use \reff{FourierM} to get
\begin{align*}
  &|D_v|^\sigma\,\biginner{H(v)\varphi_{\rho,N}(v)}=
  C_\sigma\int_{\Real^n}\frac{H(v)\varphi_{\rho,N}(v)
  -H(v-\tilde v)\varphi_{\rho,N}(v-\tilde v)}{\abs
  {\tilde v}^{n+\sigma}}\,d\tilde v\\
  &=\varphi_{\rho,N}(v)|D_v|^\sigma\, H(v)
  +C_\sigma\int_{\Real^n}\frac{
  H(v-\tilde v)\biginner{\varphi_{\rho,N}(v)
  -\varphi_{\rho,N}(v-\tilde v)}}{\abs
  {\tilde v}^{n+\sigma}}\,d\tilde v,
\end{align*}
which gives that
\begin{align}\label{formula}
  \bigcom{|D_v|^\sigma,\:\varphi_{\rho,N}(v)}H(v)=
  C_\sigma\int_{\Real^n}\frac{H(v-\tilde
 v)\biginner{\varphi_{\rho,N}(v)
  -\varphi_{\rho,N}(v-\tilde v)}}{\abs
  {\tilde v}^{n+\sigma}}\,d\tilde v.
\end{align}
Let $\widetilde\chi_{\rho/N}$ be the characteristic function of the
set $\set{v;\: \abs v\leq \rho/N}.$ By the above expression, we
compute
\begin{align*}
  &\|\bigcom{|D_v|^\sigma,\:\varphi_{\rho,N}}H\|_{L^2(\Real^n_v)}^2=
  |C_\sigma|^2\int_{\Real^n}\abs{\int_{\Real^n}\frac{H(v-\tilde
 v)\biginner{\varphi_{\rho,N}(v)
  -\varphi_{\rho,N}(v-\tilde v)}}{\abs
  {\tilde v}^{n+\sigma}}\,d\tilde v}^2dv\\
  &\leq
 2|C_\sigma|^2\int_{\Real^n}\abs{\int_{\Real^n}\frac{\widetilde{\chi}_{\rho/N}(\tilde v)
   H(v-\tilde v)\biginner{\varphi_{\rho,N}(v)
  -\varphi_{\rho,N}(v-\tilde v)}}{\abs
  {\tilde v}^{n+\sigma}}\,d\tilde v}^2dv\\
  &\indent+2|C_\sigma|^2\int_{\Real^n}\abs{\int_{\Real^n}\frac{\inner{1-
  \widetilde{\chi}_{\rho/N}(\tilde v)}
  H(v-\tilde v)\biginner{\varphi_{\rho,N}(v)
  -\varphi_{\rho,N}(v-\tilde v)}}{\abs
  {\tilde v}^{n+\sigma}}\,d\tilde v}^2dv\\
  &\leq C\biginner{\sup\abs{\partial_v\,\varphi_{\rho,N}}}^2
\int_{\Real^n}\inner{\int_{\Real^n}\frac{\widetilde{\chi}_{\rho/N}(\tilde
v)
  \abs{H(v-\tilde v)}}{\abs
  {\tilde v}^{n+\sigma-1}}\,d\tilde v}^2dv\\
  &\indent+C\biginner{\sup\abs{\varphi_{\rho,N}}}^2
\int_{\Real^n}\inner{\int_{\Real^n}\frac{\inner{1-\widetilde{\chi}_{\rho/N}(\tilde
v)}
  \abs{H(v-\tilde v)}}{\abs
  {\tilde v}^{n+\sigma}}\,d\tilde v}^2dv\\
  &=:\mc A_1+\mc A_2,
\end{align*}
For the term $\mc A_1$, Young's inequality for convolutions gives
\[
\int_{\Real^n}\inner{\int_{\Real^n}\frac{\widetilde{\chi}_{\rho/N}(\tilde
v)
  \abs{H(v-\tilde v)}}{\abs
  {\tilde v}^{n+\sigma-1}}\,d\tilde v}^2dv\leq
  \norm{H}_{L^2(\Real_v)}^2\Big\|\frac{\widetilde\chi_{\rho/N}(
 v)}{\abs
  { v}^{n+\sigma-1}}\Big\|_{L^1(\Real_v)}^2.
\]
Then \reff{cutoff} with $|\alpha|=1$ and the following inequality
\[
  \norm{\frac{\widetilde\chi_{\rho/N}( v)}{\abs
  { v}^{n+\sigma-1}}}_{L^1(\Real_v)}^2\leq
 C\left(\int_0^{\rho/N}\frac{dr}{r^{\sigma}}\right)^2
  \leq C\inner{\rho/N}^{2(1-\sigma)}
\]
deduce that
\[
  \mc A_1\leq C
  \inner{N/\rho}^{2\sigma}\norm{H}_{L^2(\Real^n_v)}^2.
\]
Similarly, we can use \reff{cutoff}  with $|\alpha|=0$ and the
inequality
\begin{align*}
 \int_{\Real^n}\inner{\int_{\Real^n}\frac{\inner{1-\widetilde{\chi}_{\rho/N}(\tilde
v)}
  \abs{H(v-\tilde v)}}{\abs
  {\tilde v}^{n+\sigma}}\,d\tilde v}^2dv&\leq
  \norm{H}_{L^2(\Real_v)}^2\Big\|\frac{1-\widetilde\chi_{\rho/N}(
 v)}{\abs
  { v}^{n+\sigma}}\Big\|_{L^1(\Real_v)}^2\\
  &\leq C\inner{\rho/N}^{-2\sigma}\norm{H}_{L^2(\Real_v)}^2
\end{align*}
to get
\[
 \mc A_2\leq C
  \inner{N/\rho}^{2\sigma}\norm{H}_{L^2(\Real^n_v)}^2.
\]
On the other hand, it is trivial to see
\[
  \|\bigcom{\left(|D_v|^\sigma-\widetilde\Lambda_v^\sigma\right),
  \:\varphi_{\rho,N}}H\|_{L^2(\Real^n_v)}
  \leq C\norm{H}_{L^2(\Real^n_v)}.
\]
Now we combine these inequalities to conclude
\begin{equation}\label{0norm}
\|\bigcom{\widetilde\Lambda_v^\sigma,\:\varphi_{\rho,N}}H\|_{L^2(\Real^n_v)}
  \leq C\inner{N/\rho}^\sigma\norm{H}_{L^2(\Real^n_v)}.
\end{equation}

Next we treat $\|\bigcom{\widetilde\Lambda_v^\sigma,
\:\varphi_{\rho,N}}H\|_{H^\kappa(\Real^n_v)}.$ Similar to the above
argument,  we study only the commutator
$\|\bigcom{|D_v|^\sigma,\:\varphi_{\rho,N}}H\|_{H^\kappa(\Real^n_v)}.$
First, we consider the case when $\kappa$ is a positive integer. Let
$\alpha$ be an arbitrary multi-index with $\abs\alpha\leq \kappa.$
Then taking derivatives in \reff{formula}, and then using Leibnitz's
formula; we get
\begin{eqnarray*}
 &&
 \partial_v^\alpha\left(\bigcom{|D_v|^\sigma,\:\varphi_{\rho,N}(v)}H(v)\right)\\
 && =
C_\sigma\sum_{\beta\leq\alpha}C_\alpha^\beta\int_{\Real^n}\frac{\left(
  \partial_v^\beta H(v-\tilde v)\right)
  \cdot\left(\partial_v^{\alpha-\beta}\biginner{\varphi_{\rho,N}(v)
  -\varphi_{\rho,N}(v-\tilde v)}\right)}{\abs
  {\tilde v}^{n+\sigma}}\,d\tilde v.
\end{eqnarray*}
Thus  similar arguments as above show that
\begin{align*}
\norm{\partial_v^\alpha\Biginner{\bigcom{|D_v|^\sigma,\:
\varphi_{\rho,N}(v)}H(v)}}_{L^2(\Real^n_v)}
  \leq
 C\sum_{\beta\leq\alpha}\inner{N/\rho}^{\abs{\alpha-\beta}+\sigma}
  \norm{\partial_v^\beta H}_{L^2(\Real^n_v)}.
\end{align*}
Together with the interpolation inequality \reff{interpolation}, we
obtain
\begin{align*}
& \norm{\partial_v^\alpha\left(\bigcom{|D_v|^\sigma,\:
\varphi_{\rho,N}(v)}H(v)\right)}_{L^2(\Real^n_v)}
 \\
 & \leq C\set{\inner{N/\rho}^\sigma\norm{H}_{H^\kappa(\Real^n_v)}
  +\inner{N/\rho}^{\abs{\alpha}+\sigma}
  \norm{H}_{L^2(\Real^n_v)}}.
\end{align*}
Since $\alpha, \abs\alpha\leq\kappa,$ is arbitrary,   we conclude
\begin{align*}
  \norm{\bigcom{\abs{D_v}^\sigma,\:\varphi_{\rho,N}(v)}H(v)}_{H^\kappa(\Real^n_v)}
  \leq
 C\set{\inner{N/\rho}^\sigma\norm{H}_{H^\kappa(\Real^n_v)}+\inner{N/\rho}^{\kappa+\sigma}
  \norm{H}_{L^2(\Real^n_v)}}.
\end{align*}
This implies \reff{+++com1}, when $\kappa$ is a positive integer.

Now we consider the case when  $\kappa$ is not a integer.  Without
loss of generality, we may assume $0<\kappa<1.$ Write
$\kappa+\sigma=1+\mu.$  Then $0\leq \mu<1,$  and
\begin{align*}
  \norm{\bigcom{\abs{D_v}^{\kappa+\sigma},\:\varphi_{\rho,N}(v)}H(v)}_{L^2(\Real^n_v)}
 &\leq
 \norm{\bigcom{\abs{D_v}^\mu,\:\varphi_{\rho,N}(v)}H(v)}_{H^1(\Real^n_v)}\\
 &\indent
 +\norm{\bigcom{\abs{D_v}^1,\:\varphi_{\rho,N}(v)}\abs{D_v}^\mu
 H(v)}_{L^2(\Real^n_v)}.
\end{align*}
We have treated the first term on the right, that is,
\[
  \norm{\bigcom{\abs{D_v}^\mu,\:\varphi_{\rho,N}(v)}H(v)}_{H^1(\Real^n_v)}\leq
  C\set{\inner{N/\rho}^\mu\norm{H}_{H^1(\Real^n_v)}+\inner{N/\rho}^{1+\mu}
  \norm{H}_{L^2(\Real^n_v)}}.
\]
On the other hand, one has
\[
  \norm{\bigcom{\abs{D_v}^1,\:\varphi_{\rho,N}(v)}\abs{D_v}^\mu
  H(v)}_{L^2(\Real^n_v)}\leq
  C\inner{N/\rho}\norm{H}_{H^\mu(\Real^n_v)}.
\]
For the proof of this estimate, we refer to \cite{Dur} for instance.
Hence
\begin{align*}
 \norm{\bigcom{\abs{D_v}^{\kappa+\sigma},\:\varphi_{\rho,N}(v)}H(v)}_{L^2(\Real^n_v)}
 &\leq
 C\Big\{\inner{N/\rho}^\mu\norm{H}_{H^1(\Real^n_v)}+\inner{N/\rho}^{1+\mu}
  \norm{H}_{L^2(\Real^n_v)}\\
  &\indent+\inner{N/\rho}\norm{H}_{H^\mu(\Real^n_v)}\Big\}.
\end{align*}
Notice that $\kappa\geq 1,$. The interpolation inequality
\reff{interpolation} gives
\begin{align*}
& \norm{\bigcom{\abs{D_v}^{\kappa+\sigma},\:\varphi_{\rho,N}(v)}H(v)}_{L^2(\Real^n_v)}\\
 &\leq
 C\Big\{\inner{N/\rho}^\sigma\norm{H}_{H^\kappa(\Real^n_v)}+\inner{N/\rho}^{\kappa+\sigma}
  \norm{H}_{L^2(\Real^n_v)}\Big\}.
\end{align*}
 Since $0<\kappa<1,$ then
\begin{align*}
& \norm{\bigcom{\abs{D_v}^{\kappa},\:\varphi_{\rho,N}(v)}
\abs{D_v}^{\sigma}H(v)}_{L^2(\Real^n_v)}\\
 &\leq
 C\Big\{\inner{N/\rho}^\kappa\norm{H}_{H^\sigma(\Real^n_v)}+\inner{N/\rho}^{\kappa+\sigma}
  \norm{H}_{L^2(\Real^n_v)}\Big\}\\
 &\leq
 C\Big\{\inner{N/\rho}^\sigma\norm{H}_{H^\kappa(\Real^n_v)}+\inner{N/\rho}^{\kappa+\sigma}
  \norm{H}_{L^2(\Real^n_v)}\Big\}.
\end{align*}
In the last inequality, we have used the interpolation inequality
\reff{interpolation}. The above two inequalities yield that
\begin{align*}
&\norm{\abs{D_v}^{\kappa}\bigcom{\abs{D_v}^{\sigma},
\:\varphi_{\rho,N}(v)}H(v)}_{L^2(\Real^n_v)}\\
&\leq
\norm{\bigcom{\abs{D_v}^{\kappa+\sigma},\:\varphi_{\rho,N}(v)}H(v)}_{L^2(\Real^n_v)}\\
&\indent+\norm{\bigcom{\abs{D_v}^{\kappa},\:
\varphi_{\rho,N}(v)}\abs{D_v}^{\sigma}H(v)}_{L^2(\Real^n_v)}\\
&\leq
 C\Big\{\inner{N/\rho}^\sigma\norm{H}_{H^\kappa(\Real^n_v)}+\inner{N/\rho}^{\kappa+\sigma}
  \norm{H}_{L^2(\Real^n_v)}\Big\}.
\end{align*}
Hence
\begin{align*}
&\norm{\bigcom{\abs{D_v}^{\sigma},\:\varphi_{\rho,N}(v)}H(v)}_{H^\kappa(\Real^n_v)}\\
&\leq C\Big\{\norm{\abs{D_v}^{\kappa}\bigcom{\abs{D_v}^{\sigma},\:
\varphi_{\rho,N}(v)}H(v)}_{L^2(\Real^n_v)}\\
&\indent+\norm{\bigcom{\abs{D_v}^{\sigma},\:
\varphi_{\rho,N}(v)}H(v)}_{L^2(\Real^n_v)}\Big\}\\
&\leq
 C\Big\{\inner{N/\rho}^\sigma\norm{H}_{H^\kappa(\Real^n_v)}+\inner{N/\rho}^{\kappa+\sigma}
  \norm{H}_{L^2(\Real^n_v)}\Big\}.
\end{align*}
This implies \reff{+++com1} for general $\kappa, 1\leq\kappa\leq
n+2,$ and thus \reff{+com1} follows. The inequality \reff{+com2} can
be handled quite similarly. Thus the proof of Lemma \ref{+lem9} is
complete.

\end{proof}


\section{Gevrey regularity of linear operators}
\label{sect3+} \setcounter{equation}{0}

In this section, we prove  the Gevrey hypoellipticity of $\mc P$. We
will follow the idea of M.Durand \cite{Dur}. We consider the
following linear equation
\begin{equation}\label{Fokker-Planck++}
\mc{P}u=\partial_t u+v\cdot\partial_x
u+a(t,x,v)(-\widetilde\triangle_v)^\sigma u=f,\quad(t,x,
v)\in{\Real\times\Real^n\times\bb R^{n}},
\end{equation}
where $ 0<\sigma<1$. From Theorem \ref{th2.1}, any weak solution of
the above equation is in $C^\infty(\bb R^{2n+1})$ if $f\in
C^\infty(\bb R^{2n+1})$ . Hence, we start from a $C^\infty$
solution, and prove the Gevrey hypoellipticity in the following
proposition, where $\Omega$ and  $W=2\Omega$ are open domains of
$\mathbb{R}^{2n+1}$ defined in the section \ref{sect+3}.

\begin{prp}\label{prp4}
Set $\delta=\max\set{{\sigma\over4},~{\sigma\over2}-{1\over6}}$ and
let $s\geq {2\over\delta}$. Suppose the coefficient $ a(t,x,v) \in
G^s(\bar\Omega), a>0,$ and $u\in C^\infty(\bar W)$ be such that
$\mc{P}u=f\in G^s(\bar{\Omega})$. Then there exits a constant $L$
such that for any $r\in[0,1]$ and any $N\in\bb{N}$, $N\geq4,$
$$
(E)_{r, N} \quad\quad
\begin{array} {l}\|\Phi_{\rho,N}D^\alpha
u\|_{r+n+1}+\|\Phi_{\rho,N}{\widetilde\Lambda}^{\sigma} D^\alpha
u\|_{r-\frac{\delta}{2}+n+1}\\
\hskip 1cm \leq \frac{L^{|\alpha|-1}}{\rho^{(s+n)(|\alpha|-3)}}
\big((|\alpha|-3)!\big)^{s}\inner{\frac{N}{\rho}}^{sr}
\end{array}
$$
holds for any $\alpha\in \bb{N}^{2n+1},\:\:|\alpha|= N$ and  any $
0<\rho< 1$.
\end{prp}

\begin{rmk}
 Here the Gevrey constant $L$ of $u$ is  determined by the
Gevrey constants $B_a$ and $B_f$ of the functions $a, f\in
G^s(\bar{\Omega})$, and depends only on $s, \sigma, n,
\norm{u}_{H^{n+6}(W)}$ and $\norm{a}_{C^{2n+2}(\Omega)}.$ This can
be seen in the  proof.
\end{rmk}

\bigbreak  As an immediate consequence of the above  proposition, we
have

\begin{prp}\label{prp5}

Under the same assumption as in Proposition \ref{prp4},  we have
$u\in G^s({\Omega}).$
\end{prp}

Indeed, for any compact subset $K$ of $\Omega$, we have
$K\subset\Omega_{\rho_0}$ for some $\rho_0, ~0<\rho_0<1$. Then for
any $\alpha\in {\bb N}^{2n+1}, ~~|\alpha|=N\geq4$,  $(E)_{0, N}$
gives
\begin{equation*}
\begin{array}{lll}\|D^\alpha u\|_{L^2(K)}&\leq&\|\Phi_{\rho_0,N}D^\alpha
u\|_{n+1} \leq
\frac{L^{|\alpha|-1}}{{\rho_0}^{(s+n)(|\alpha|-3)}}\big((|\alpha|-3)!\big)^{s}
\leq \big({L\over{\rho_0}^{s+n}}\big)^{|\alpha|}(|\alpha|!)^s.
\end{array}
\end{equation*}
Taking $C_K={L\over{\rho_0}^{s+n}}+\norm u_{C^4(K)},$ then for all
$\alpha,$
\[
  \|D^\alpha u\|_{L^2(K)}\leq C_K^{\abs\alpha+1}(|\alpha|!)^s.
\]
The conclusion of Proposition \ref{prp5} follows.

\bigskip

\smallbreak \emph{Proof of Proposition \ref{prp4}.}  We prove the
esitimate $(E)_{r, N}$ by induction on $N$. In the proof, we use
$C_n$ to denote  constants which depend only on $n$, which may be
different in different contexts. Let $\Phi$ be an arbitrary fixed
function compactly supported in $W$ such that $\Phi=1$ in $\Omega.$
First, we prove the first step of the induction for $N=4$. For all
$\abs\alpha=4,$ we use \reff{cutoffnorm} in Lemma \ref{3.1} to
compute
\begin{align*}
    &\|\Phi_{\rho,3}D^\alpha u\|_{r+n+1}+
    \|\Phi_{\rho,3}{\widetilde\Lambda}^{\sigma} D^\alpha
    u\|_{r-\frac{\delta}{2}+n+1}\\
    &\leq C_n \inner{\frac{3}{\rho}}^{n+2}\set{\|\Phi D^\alpha u\|_{r+n+1}+
    \|\Phi {\widetilde\Lambda}^{\sigma} D^\alpha
    u\|_{r-\frac{\delta}{2}+n+1}},
\end{align*}
On the other hand , since $\abs\alpha=4,$
\[
   \|\Phi D^\alpha u\|_{r+n+1}+
    \|\Phi {\widetilde\Lambda}^{\sigma} D^\alpha
    u\|_{r-\frac{\delta}{2}+n+1}\leq C_n \|u\|_{H^{n+6}(W)}.
\]
The term on the left side is bounded by the smoothness of $u.$
Combing these, we obtain

\[
  (E)_{r, 4} \qquad  \|\Phi_{\rho,3}D^\alpha u\|_{r+n+1}+
  \|\Phi_{\rho,3}{\widetilde\Lambda}^{\sigma} D^\alpha
u\|_{r-\frac{\delta}{2}+n+1} \leq
\frac{C_n\|u\|_{H^{n+6}(W)}}{\rho^{(n+2)}}\leq
\frac{L^3_0}{\rho^{s+n}}\,\,.
\]
Thus $(E)_{r, 4}$ is true if we take $L\geq
C_n\|u\|_{H^{n+6}(W)}+1.$ Let now $N>4$ and assume that $(E)_{r,
N-1}$ holds for any $r\in[0, 1]$. We  need to show $(E)_{r, N}$
still holds with a constant $L$ independents of $N$ or $r\in[0,~1]$.
We denote
\[
\|D^ju\|_r=\sum_{|\gamma|=j}\|D^\gamma u\|_r.
\] In the
following discussion, we fix $N.$ For each $0<\rho<1,$ define
$\tilde\rho={{N-1}\over N}\rho, \: \tilde{\tilde\rho}={{N-2}\over
N}\rho.$ Let $\Phi_{\rho,N}$ be the cutoff function constructed in
the previous section which satisfies the property \reff {cutoff}.
The following fact will be used frequently
\begin{equation}\label{fre}
{1\over{\rho}^{(s+n)k}}\leq{1\over{\tilde\rho}^{(s+n)k}}
\leq{1\over{\tilde{\tilde\rho}}^{(s+n)k}} ={1\over{\rho}^{(s+n)k}}
\times\big({N\over{N-2}}\big)^{(s+n)k} \leq
{10^{s+n}\over{\rho}^{sk}},\quad k=1,2,\cdots, N.
\end{equation}

We shall proceed to prove the truth of $(E)_{r, N}$  by the
following four lemmas. The first one is a technical lemma, and the
second lemma is devote to the proof of the truth of $(E)_{r, N}$ for
$r=0.$ In the third one, we prove that $(E)_{r,N}$ holds for $0\leq
r\leq{\delta\over 2},$ and in the last one we prove that $(E)_{r,
N}$ holds for all $r$ with $0\leq r\leq1.$

\begin{lem}\label{tech}
  Let $s\geq 3$ be a given real number and  $k\geq 5$ be any given integer. Assume
  \begin{equation}\label{assum}
    \norm{\Phi_{\rho,m}D^\gamma u}_{n+1}+
    \norm{\Phi_{\rho,m}\widetilde \Lambda^\sigma D^\gamma
    u}_{-\frac{\delta}{2}+n+1}
    \leq\frac{L^{m-1}}{\rho^{(s+n)(m-3)}}\big((m-3)!\big)^{s}
  \end{equation}
holds for all $\gamma$ with $\abs\gamma=m<k,$ and all $0<\rho<1.$
Then if $L\geq 4^{n+3}(\norm{u}_{H^{n+6}(W)}+1),$ one has, for all
$\beta$ with $\abs\beta=k,$
  \begin{equation}\label{desired}
    (k/\rho)^{n+3}\norm{\Phi_{\rho,k}D^\beta u}_{0}+(k/\rho)^{n+3}
    \norm{\Phi_{\rho,k}\widetilde \Lambda^\sigma D^\beta
    u}_{0}
    \leq\frac{L^{k-2}}{\rho^{(s+n)(k-3)}}\big((k-3)!\big)^{s}.
  \end{equation}
\end{lem}

\begin{proof}
  Without loss of generality, we may assume $k>n+4,$ for, otherwise, in the case when
  $5\leq k\leq n+4 $,
  it is obvious that, for all $\beta$ with $\abs\beta=k\leq n+4,$
  \begin{eqnarray*}
   & (k/\rho)^{n+3}\norm{\Phi_{\rho,k}D^\beta u}_{0}+(k/\rho)^{n+3}
    \norm{\Phi_{\rho,k}\widetilde \Lambda^\sigma D^\beta
    u}_{0}&\\
   & \leq (1/\rho)^{(s+n)(k-3)}2^{(n+3)k}\norm{u}_{H^{n+6}(W)}.&
  \end{eqnarray*}
  Then the desired inequality \reff{desired} follows if $L\geq
  4^{n+3}\inner{\norm{u}_{H^{n+6}(W)}+1}.$

  \vspace{0.5ex}
  Now for all $\beta, \abs\beta=k>n+4,$ we can find a
  multi-index $\tilde\beta\leq\beta$ such that $|\tilde\beta|=n+1.$
  First we treat $(k/\rho)^{n+3}\norm{\Phi_{\rho,k}D^\beta u}_{0}.$
  Since $\Phi_{\frac{(k-1)\rho}{k}, k-n-1}=1$ in Supp\,$\Phi_{\rho,
  k},$ then the following relation is clear:
  \[
    \norm{\Phi_{\rho,k}D^\beta u}_{0}=\norm{\Phi_{\rho,k}
    D^{\tilde\beta}\Phi_{\frac{(k-1)\rho}{k}, k-n-1}D^{\beta-\tilde\beta} u}_{0}
    \leq \norm{
    \Phi_{\frac{(k-1)\rho}{k}, k-n-1}D^{\beta-\tilde\beta} u}_{n+1}.
  \]
  Observe $\abs{\beta-\tilde\beta}=k-n-1,$ then we use the above relation and the
  assumption \reff{assum} to compute, for $L\geq
  4^{n+3}\inner{\norm{u}_{H^{n+6}(W)}+1},$
  \begin{align*}
    (k/\rho)^{n+3}\norm{\Phi_{\rho,k}D^\beta u}_{0}&\leq (k/\rho)^{n+3}\norm{
    \Phi_{\frac{(k-1)\rho}{k}, k-n-1}D^{\beta-\tilde\beta}
    u}_{n+1}\\
    &\leq (k/\rho)^{n+3}\frac{L^{k-n-2}}{\rho^{(s+n)(k-n-4)}}\big((k-n-4)!\big)^{s}
    \\
    &\leq \frac{5(n/L)^{n}L^{k-2}}{\rho^{(s+n)(k-3)}}\big((k-3)!\big)^{s}
    \\
    &\leq {\frac 1 2}\frac{L^{k-2}}{\rho^{(s+n)(k-3)}}\big((k-3)!\big)^{s}
    .
  \end{align*}
  In the same way, we can get the estimate on the term  $(k/\rho)^{n+3}
    \norm{\Phi_{\rho,k}\widetilde \Lambda^\sigma D^\beta
    u}_{0},$ that is,
    \begin{align*}
    (k/\rho)^{n+3}
    \norm{\Phi_{\rho,k}\widetilde \Lambda^\sigma D^\beta
    u}_{0}\leq {\frac 1 2}\frac{L^{k-2}}{\rho^{(s+n)(k-3)}}\big((k-3)!\big)^{s}
    .
  \end{align*}
  Thus by the above two inequalities, we get the desired inequality
  \reff{desired}. This completes the proof.
\end{proof}

\begin{lem}\label{r0}
Assume that $(E)_{r, N-1}$ is true for any $r\in[0, 1]$.  Then there
exists a constant $C_1,$ depending only on the Gevrey index $s$ and
the dimension $n$, such that, if $ L\geq \max\Big\{2^{s+1}B_a,\,
B_f, \,4^{s+1}\norm{u}_{H^{n+6}(W)}+4^{s+1}\Big\},$
\begin{equation}\label{rho0}
\|\Phi_{\rho,N}D^\alpha
u\|_{n+1}+\|\Phi_{\rho,N}{\widetilde\Lambda}^{\sigma} D^\alpha
u\|_{-{\delta\over2}+n+1}\leq \frac{C_1
L^{|\alpha|-2}}{\rho^{(s+n)(|\alpha|-3)}}\big((|\alpha|-3)!\big)^{s}
\end{equation}
for any $\alpha\in \bb{N}^{2n+1},\:\:|\alpha|= N, $ and any
$0<\rho<1$.
\end{lem}

\begin{rmk} In fact, this is $(E)_{r, N}$ for $r=0$
if we choose $L$ such that  $ L\geq C_1$ and $L\geq
\max\Big\{2^{s+1}B_a,\, B_f,\,
4^{s+1}\norm{u}_{H^{n+6}(W)}+4^{s+1}\Big\}$.
\end{rmk}

\begin{proof}

We choose a multi-index $\beta$ with $\abs\alpha=\abs\beta+1$. Then
$|\beta|=N-1$. Recall $\tilde\rho={{N-1}\over N}\rho.$  By the
construction, $\Phi_{\tilde\rho, N-1}=1$ in Supp\,$\Phi_{\rho,N}.$
Thus
\begin{align*}
&\|\Phi_{\rho,N}D^\alpha u\|_{n+1} \leq\|\Phi_{\rho, N}D^\beta
u\|_{1+n+1}+\|(D\Phi_{\rho, N})D^\beta u\|_{n+1}\\
&\leq\|\Phi_{\rho, N}\Phi_{\tilde\rho, N-1}D^\beta
u\|_{1+n+1}+\|(D\Phi_{\rho, N})\Phi_{\tilde\rho, N-1}D^\beta
u\|_{n+1}\\&\leq C_n\set{\|\Phi_{\tilde\rho, N-1}D^\beta
u\|_{1+n+1}+(N/\rho)\|\Phi_{\tilde\rho, N-1}D^\beta
u\|_{n+1}+(N/\rho)^{n+2}\|\Phi_{\tilde\rho, N-1}D^\beta u\|_{0}},
\end{align*}
In the last inequality, we have used  Lemma \ref{3.1}. For the third
term on the right-hand side, we use Lemma \ref{tech} with $k=N-1$ to
obtain
\begin{align*}
  (N/\rho)^{n+2}\|\Phi_{\tilde\rho, N-1}D^\beta u\|_{0}&=
  {{N-1}\over{\tilde\rho}}\set{\inner{\frac{N-1}{\tilde\rho}}^{n+1}
  \|\Phi_{\tilde\rho, N-1}D^\beta u\|_{0}}\\
  &\leq {{|\beta|}\over{\tilde\rho}}
  \frac{L^{|\beta|-2}}{\tilde\rho^{(s+n)(|\beta|-3)}}\big((|\beta|-3)!\big)^{s}\\
  &\leq
  \frac{2L^{|\alpha|-2}}{\tilde\rho^{(s+n)(|\alpha|-3)}}\big((|\alpha|-3)!\big)^{s}.
\end{align*}
Applying the relation \reff{fre}, we get
\begin{align*}
  (N/\rho)^{n+2}\|\Phi_{\tilde\rho, N-1}D^\beta u\|_{0}\leq
  \frac{20^{s+n}L^{|\alpha|-2}}{\rho^{(s+n)(|\alpha|-3)}}\big((|\alpha|-3)!\big)^{s}.
\end{align*}

On the other hand, by the induction assumption that $(E)_{r,N-1}$
holds for any $r$ with $0\leq r\leq1$, we have immediately
\begin{align*}
  &\|\Phi_{\tilde\rho, N-1}D^\beta
u\|_{1+n+1}+(N/\rho)\|\Phi_{\tilde\rho, N-1}D^\beta
u\|_{n+1}\\
  &\leq \frac{L^{|\beta|-1}}{\tilde\rho^{(s+n)(|\beta|-3)}}
     \big((|\beta|-3)!\big)^{s}(N/\tilde\rho)^{s}
 +(N/\rho)\frac{L^{|\beta|-1}}{\tilde\rho^{(s+n)(|\beta|-3)}}\big((|\beta|-3)!\big)^{s}\\
  &\leq \frac{2L^{|\alpha|-2}}{{\tilde\rho}^{(s+n)(|\alpha|-3)}}
     \big((|\alpha|-3)!\big)^{s}\big(N/(N-3)\big)^{s}\\
  &\leq\frac{30^{s+n}L^{|\alpha|-2}}{\rho^{(s+n)(|\alpha|-3)}}\big((|\alpha|-3)!\big)^{s}.
\end{align*}
Thus
\begin{equation*}
\|\Phi_{\rho,N}D^\alpha u\|_{n+1}\leq \frac{30^{s+n} C_n
L^{|\alpha|-2}}{\rho^{(s+n)(|\alpha|-3)}}\big((|\alpha|-3)!\big)^{s}.
\end{equation*}
By  exactly the same calculation, we obtain
\begin{eqnarray*}
\|\Phi_{\rho,N}{\widetilde\Lambda}^{\sigma} D^\alpha
u\|_{-{\delta\over2}+n+1} \leq \frac{30^{s+n}
C_nL^{|\alpha|-2}}{\rho^{(s+n)(|\alpha|-3)}}\big((|\alpha|-3)!\big)^{s}.
\end{eqnarray*}
Taking $C_1=60^{s+n} C_n$ with $C_n$ being the constant appearing in
Lemma \ref{3.1},  we obtain \reff{rho0}. This completes the proof of
Lemma \ref{r0}.

\end{proof}

\begin{lem}\label{r1/3}
Assume that $(E)_{r, N-1}$ is true for any $r\in[0, 1]$. Then there
exists a constant $C_2,$ depending only on $\sigma, $ the Gevrey
index $s,$ the dimension $n$ and $\norm{u}_{H^{n+6}(W)},
\norm{a}_{C^{n+2}(\bar \Omega)},$ such that for any $0\leq r\leq
\frac{\delta}{2}$, if
$$
L\geq \max\Big\{2^{s+1}B_a,\, B_f,
\,4^{s+1}\norm{u}_{H^{n+6}(W)}+4^{s+1}\Big\}
$$
with $B_a, B_f$ being the Gevrey constants of $a, f\in
G^s(\bar\Omega)$, we have that
\begin{equation}\label{r}
\|\Phi_{\rho,N}D^\alpha
u\|_{r+n+1}+\|\Phi_{\rho,N}{\widetilde\Lambda}^{\sigma} D^\alpha
u\|_{r-\frac{\delta}{2}+n+1}\leq
\frac{C_{2}L^{|\alpha|-2}}{\rho^{(s+n)(|\alpha|-3)}}\big((|\alpha|-3)!\big)^{s}(N/\rho)^{rs},
\end{equation}
for any $\alpha\in \bb{N}^{2n+1},\:\:|\alpha|= N$.
\end{lem}

\begin{proof}

In this proof, we shall use $\widetilde C_j, j\geq0,$ to denote
different constants which are greater than 1 and depend  only on $s,
\sigma, n, \norm{u}_{H^{n+6}(W)}$ and $\norm{a}_{C^{2n+2}(\Omega)}.$
The conclusion  will follow  if we prove that
\begin{equation}\label{r-theta}
  \|\Phi_{\rho, N}D^\alpha u\|_{\frac{\delta}{2}+n+1}+\|\Phi_{\rho,
 N}{\widetilde\Lambda}^{\sigma} D^\alpha u\|_{n+1}
  \leq\frac{\widetilde
 C_{0}L^{|\alpha|-2}}{\rho^{(s+n)(|\alpha|-3)}}\big((|\alpha|-3)!\big)^{s}
     (N/\rho)^{\frac{s\delta}{2}}.
\end{equation}
Indeed, from \reff{r-theta} we know that \reff{r} is true for
$r={\delta\over2}.$    The truth of \reff{r} for the general $r,
0\leq r\leq \frac{\delta}{2},$ follows  from the interpolation
inequality (\ref{interpolation}) and Lemma \ref{r0}.

To prove \reff{r-theta}, we shall proceed in the following four
steps.

\bigskip {\bf Step 1.} In this step we prove
\begin{equation}
\label{step+1}
 \|a[{\widetilde\Lambda}^{2\sigma},~~\Phi_{\rho,N}D^\alpha]u\|_{-\frac{\delta}{2}+ n+1}
    \leq\frac{\widetilde C_1 L^{|\alpha|-2}}{\rho^{(s+n)(|\alpha|-3)}}
    \big((|\alpha|-3)!\big)^{s}(N/\rho)^{\frac{s\delta}{2}}.
\end{equation}

Recall $\Phi_{\rho,N}(t,x,v)=\psi_{\rho, N}(t,x)\vpi_{\rho, N}(v)$
with $\psi_{\rho, N},\vpi_{\rho, N}$ being the cut-off functions
constructed in Section 3.  First, notice that $\psi_{\tilde\rho,
N}=1$ in the support of \:$\psi_{\rho,N},$ and $\vpi_{\tilde\rho,
N}=1$ in the support of \:$\vpi_{\rho,N}$. It then follows that
 \begin{align*}
 &\|a[{\widetilde\Lambda}^{2\sigma},~~\Phi_{\rho,N}D^\alpha]
 u\|_{-\frac{\delta}{2}+
 n+1}=\|a[{\widetilde\Lambda}^{2\sigma},~~\varphi_{\rho,N}]\psi_{\rho,N}D^\alpha
 u\|_{-\frac{\delta}{2}+
 n+1}\\
 &\leq C_a\set{
 \|[{\widetilde\Lambda}^{\sigma},~~\varphi_{\rho,N}]\psi_{\rho,N}
 {\widetilde\Lambda}^{\sigma}D^\alpha u\|_{-\frac{\delta}{2}+
 n+1}+\|[{\widetilde\Lambda}^{\sigma},\:
 [{\widetilde\Lambda}^{\sigma},~~\varphi_{\rho,N}]\:]\psi_{\rho,N}D^\alpha u\|_{-\frac{\delta}{2}+
 n+1}}\\
 &\leq C_a\Big\{
 \|[{\widetilde\Lambda}^{\sigma},~~\varphi_{\rho,N}]\psi_{\rho,N}\,\psi_{\tilde\rho,N}\,
 \vpi_{\tilde\rho,N}
 {\widetilde\Lambda}^{\sigma}D^\alpha u\|_{-\frac{\delta}{2}+
 n+1}\\
 &\indent+\|[{\widetilde\Lambda}^{\sigma},\:
 [{\widetilde\Lambda}^{\sigma},~~\varphi_{\rho,N}]\:]
 \psi_{\rho,N}\,\psi_{\tilde\rho,N}\,\vpi_{\tilde\rho,N}D^\alpha u\|_{-\frac{\delta}{2}+
 n+1}\Big\}\\
 &= C_a\Big\{
 \|[{\widetilde\Lambda}^{\sigma},~~\varphi_{\rho,N}]\psi_{\rho,N}\Phi_{\tilde\rho,N}
 {\widetilde\Lambda}^{\sigma}D^\alpha u\|_{-\frac{\delta}{2}+
 n+1}\\
 & \,\,\,\,\,\,\,\, +\|[{\widetilde\Lambda}^{\sigma},\:
 [{\widetilde\Lambda}^{\sigma},~~\varphi_{\rho,N}]\:]
 \psi_{\rho,N}\Phi_{\tilde\rho,N}D^\alpha u\|_{-\frac{\delta}{2}+
 n+1}\Big\}\\
 &=:(S_1)+(S_2),
 \end{align*}
where $C_a$  is a constants depending only on the coefficient $a$
through $\norm a_{C^{n+2}(\bar\Omega)}.$ To estimate the term
$(S_1)$, we apply the inequality \reff{+com1} in Lemma \ref{+lem9}
and then \reff{cutoffnorm} in Lemma \ref{3.1}. This gives
 \begin{align*}
 (S_1)&\leq
 C_aC_{\sigma,n}\Big\{\inner{N/\rho}^\sigma\norm{\psi_{\rho,N}\Phi_{\tilde\rho,
 N}{\widetilde\Lambda}^{\sigma} D^\alpha u}_{-\frac{\delta}{2}+
 n+1}\\
 &\hskip 1cm +\inner{N/\rho}^{n+1-\frac{\delta}{2}+\sigma}
 \norm{\psi_{\rho,N}\Phi_{\tilde\rho, N}{\widetilde\Lambda}^{\sigma}
 D^\alpha
 u}_{0}\Big\}\\
 &\leq
 C_aC_{\sigma,n}\set{\inner{N/\rho}^\sigma\norm{\Phi_{\tilde\rho,
 N}{\widetilde\Lambda}^{\sigma} D^\alpha u}_{-\frac{\delta}{2}+
 n+1}+\inner{N/\rho}^{n+1-\frac{\delta}{2}+\sigma}\norm{\Phi_{\tilde\rho, N}
 {\widetilde\Lambda}^{\sigma}
 D^\alpha
 u}_{0}}\\
 &=:(S_1)^{'}+(S_1)^{''}.
 \end{align*}
First, the estimate \reff{rho0} in Lemma \ref{r0} yields
\begin{align*}
(S_1)'&\leq C_aC_{\sigma,n}C_1\inner{{N\over\rho}}^\sigma \frac{
L^{|\alpha|-2}}{{\tilde\rho}^{(s+n)(|\alpha|-3)}}\big((|\alpha|-3)!\big)^{s}\\
&\leq \frac{\widetilde C_2
L^{|\alpha|-2}}{\rho^{(s+n)(|\alpha|-3)}}\big((|\alpha|-3)!\big)^{s}
\inner{{N\over\rho}}^{\frac{s\delta}{2}}.
\end{align*}
In the last inequality, we used the fact
$\frac{s\delta}{2}\geq1>\sigma.$  Next, we treat $(S_2)^{''}.$ By
virtue of the induction assumption,  the required condition
\reff{assum} in Lemma \ref{tech} is satisfied with $k=N$. It thus
follows from \reff{desired} that
 \begin{align*}
 (S_1)^{''}&\leq C_aC_{\sigma,n}\inner{\frac{N}{\rho}}^\sigma\frac{
 L^{|\alpha|-2}}{\tilde\rho^{(s+n)(|\alpha|-3)}}\big((|\alpha|-3)!\big)^{s}\\
 &\leq \frac{\widetilde C_2
 L^{|\alpha|-2}}{\rho^{(s+n)(|\alpha|-3)}}\big((|\alpha|-3)!\big)^{s}
 \inner{{N\over\rho}}^{\frac{s\delta}{2}}.
 \end{align*}
Thus
\begin{equation*}
(S_1) \leq \frac{\widetilde C_3
 L^{|\alpha|-2}}{\rho^{(s+n)(|\alpha|-3)}}\big((|\alpha|-3)!\big)^{s}
 \inner{{N\over\rho}}^{\frac{s\delta}{2}}.
\end{equation*}

\vspace{1ex} Now it remain to treat the term $(S_2).$  By the
similar arguments as above, the inequality \reff{+com2} in Lemma
\ref{+lem9} gives
\begin{align*}
 (S_2)\leq  \widetilde
 C_4\inner{N/\rho}^{2\sigma}\norm{\Phi_{\tilde\rho, N}D^\alpha
 u}_{-\frac{\delta}{2}+
 n+1}+\widetilde C_4
 \inner{N/\rho}^{n+1-\frac{\delta}{2}+2\sigma}\norm{\Phi_{\tilde\rho,
 N}D^\alpha u}_{0}:=\mc N_1+\mc N_2.
 \end{align*}
We first estimate $\mc N_1.$ Choose a multi-index $\beta$ with
$\abs\alpha=\abs\beta+1.$ Then the similar arguments as the proof of
Lemma \ref{r0} give
 \begin{align*}
   \norm{\Phi_{\tilde\rho, N}D^\alpha u}_{-\frac{\delta}{2}+n+1}
   &\leq C_n\Big\{\|\Phi_{\tilde{\tilde\rho}, N-1}D^\beta
   u\|_{(1-\frac{\delta}{2})+n+1}\\
   &+(N/\rho)\|\Phi_{\tilde{\tilde\rho}, N-1}D^\beta
   u\|_{-\frac{\delta}{2}+n+1}+(N/\rho)^{n+2-{\delta\over2}}
   \|\Phi_{\tilde{\tilde\rho}, N-1}D^\beta
   u\|_{0}\Big\}.
 \end{align*}
 We recall $\tilde{\tilde\rho}=\frac{(N-2)\rho}{N}.$
 By the interpolation inequality \reff{interpolation},
 \begin{align*}
  (N/\rho)\|\Phi_{\tilde{\tilde\rho}, N-1}D^\beta
   u\|_{-\frac{\delta}{2}+n+1}\leq  \|\Phi_{\tilde{\tilde\rho}, N-1}D^\beta
   u\|_{(1-\frac{\delta}{2})+n+1}+\inner{{N\over\rho}}^{n+2-{\delta\over2}}
   \|\Phi_{\tilde{\tilde\rho}, N-1}D^\beta u\|_{0}\Big\}.
 \end{align*}
Therefore
 \begin{align*}
   \norm{\Phi_{\tilde\rho, N}D^\alpha u}_{-\frac{\delta}{2}+n+1}
   &\leq C_n\Big\{\|\Phi_{\tilde{\tilde\rho}, N-1}D^\beta
   u\|_{(1-\frac{\delta}{2})+n+1}+(N/\rho)^{n+2-{\delta\over2}}
   \|\Phi_{\tilde{\tilde\rho}, N-1}D^\beta u\|_{0}\Big\}
 \end{align*}
Hence $\mc N_1\leq \mc N_{1,1}+\mc N_{1,2}$  with $\mc N_{1,1},\,
\mc N_{1,2}$ given by
\[
  \mc N_{1,1}=\widetilde C_5 \inner{\frac{N}{\rho}}^{2\sigma}
  \|\Phi_{\tilde{\tilde\rho}, N-1}D^\beta
   u\|_{(1-\frac{\delta}{2})+n+1},
   \quad   \mc N_{1,2}=\widetilde C_5
   \inner{\frac{N}{\rho}}^{n+2-{\delta\over2}+2\sigma} \|\Phi_{\tilde{\tilde\rho}, N-1}D^\beta
   u\|_{0}.
\]
Since $(E)_{r,N-1}$ holds for all $r\in[0,r],$ then it follows that
\begin{align*}
  \mc N_{1,1}&\leq \widetilde C_5\nrho^{2\sigma}
  \frac{L^{|\alpha|-2}}{\tilde{\tilde\rho}^{(s+n)(|\alpha|-4)}}\big((|\alpha|-4)!\big)^{s}
  \inner{\frac{N-1}{\tilde{\tilde\rho}}}^{s(1-\frac{\delta}{2})}
  \\
  &\leq \widetilde C_6\nrho^{2\sigma-\frac{s\delta}{2}}
  \frac{L^{|\alpha|-2}}{\rho^{(s+n)(|\alpha|-4)}}\big((|\alpha|-4)!\big)^{s}
  \inner{\frac{N-3}{\rho}}^{s}
  \\
  &\leq \widetilde C_6\nrho^{\frac{s\delta}{2}}
  \frac{L^{|\alpha|-2}}{\rho^{(s+n)(|\alpha|-3)}}\big((|\alpha|-3)!\big)^{s}
  .
\end{align*}
In the last inequality, we used again the fact
$\frac{s\delta}{2}\geq\sigma.$  For the  term $\mc N_{1,2}$, we use
Lemma \ref{tech} with $k=N-1.$ This gives
\begin{align*}
  \mc N_{1,2}&\leq \widetilde C_5
   \inner{\frac{N-2}{\tilde{\tilde\rho}}}^{n+2-{\delta\over2}+2\sigma}
   \inner{\frac{N-1}{\tilde{\tilde\rho}}}^{-(n+1)}\set{
   \inner{\frac{N-1}{\tilde{\tilde\rho}}}^{(n+1)} \|\Phi_{\tilde{\tilde\rho}, N-1}D^\beta
   u\|_{0}}\\
   &\leq \widetilde C_5
   \inner{\frac{N-1}{\tilde{\tilde\rho}}}^{1-{\delta\over2}+2\sigma}
  \frac{L^{|\beta|-2}}{\tilde{\tilde\rho}^{(s+n)(|\beta|-3)}}\big((|\beta|-3)!\big)^{s}.
\end{align*}
Since $1-{\delta\over2}+2\sigma<s,$ then it follows from the above
inequality that
\begin{align*}
  \mc N_{1,2}\leq
  \frac{\widetilde C_7L^{|\alpha|-2}}{{\rho}^{(s+n)(|\alpha|-3)}}\big((|\alpha|-3)!\big)^{s}.
\end{align*}
With the estimate on $\mc N_{1,1},$ one has
\begin{align*}
  \mc N_{1}=\mc N_{1,2}+\mc N_{1,2}\leq
  \frac{\widetilde C_8L^{|\alpha|-2}}{\rho^{(s+n)(|\alpha|-3)}}\big((|\alpha|-3)!\big)^{s}
  \nrho^{\frac{s\delta}{2}}
  .
\end{align*}

\vspace{0.5ex}
 In the following,  we treat $\mc N_2=\widetilde C_4
 \inner{N/\rho}^{n+1-\frac{\delta}{2}+2\sigma}\norm{\Phi_{\tilde\rho,
 N}D^\alpha u}_{0}.$ Using Lemma \ref{tech} with $k=N,$ we get
\begin{align*}
   \mc N_{2}&\leq \widetilde C_4
   \inner{\frac{N}{\rho}}^{n+2-{\delta\over2}+2\sigma}
   \inner{\frac{N}{{\tilde\rho}}}^{-(n+3)}\set{
   \inner{\frac{N}{{\tilde\rho}}}^{(n+3)} \|\Phi_{{\tilde\rho},
   N}D^\alpha
   u\|_{0}}\\
   &\leq \widetilde C_4
   \inner{\frac{N}{{\tilde\rho}}}^{\sigma}
  \frac{L^{|\alpha|-2}}{{\tilde\rho}^{(s+n)(|\alpha|-3)}}\big((|\alpha|-3)!\big)^{s}\\
  &\leq
  \frac{\widetilde C_{9}L^{|\alpha|-2}}{\rho^{(s+n)(|\alpha|-3)}}\big((|\alpha|-3)!\big)^{s}
  \inner{\frac{N}{{\rho}}}^{{{s\delta}\over2}}.
\end{align*}
Thus,
\begin{align*}
  (S_2)=\mc N_{1}+\mc N_{2}\leq
  \frac{\widetilde C_{10}
  L^{|\alpha|-2}}{\rho^{(s+n)(|\alpha|-3)}}\big((|\alpha|-3)!\big)^{s}
  \nrho^{\frac{s\delta}{2}}
  .
\end{align*}
With the estimate on $(S_1),$ we get the desired inequality
\reff{step+1}. This completes the proof of Step 1.

\vspace{1ex} {\bf Step 2.} In this step, we prove
\begin{equation}
\label{step1}
   \|[\p,~~ \Phi_{\rho,N}D^\alpha] u\|_{-\frac{\delta}{2}+n+1}
    \leq\frac{\widetilde C_{11}L^{|\alpha|-2}}{\rho^{(s+n)(|\alpha|-3)}}
    \big((|\alpha|-3)!\big)^{s}(N/\rho)^{\frac{s\delta}{2}}.
\end{equation}

\smallskip
Recall $\p=X_0+a{\widetilde\Lambda}^{2\sigma}$ with
$X_0=\partial_t+v\cdot\partial_x$. Then a direct computation deduces
that
\begin{eqnarray*}
\|[\p,~~ \Phi_{\rho,N}D^\alpha] u\|_{-\frac{\delta}{2}+
n+1}&\leq&\|[X_0,~~ \Phi_{\rho,N}D^\alpha]
u\|_{-\frac{\delta}{2}+n+1}+\|a[{\widetilde\Lambda}^{2\sigma},~~
\Phi_{\rho,N}D^\alpha] u\|_{-\frac{\delta}{2}+
n+1}\\&&+\|\Phi_{\rho,N}[a,~~
D^\alpha]{\widetilde\Lambda}^{2\sigma} u\|_{-\frac{\delta}{2}+n+1}\\
&=:&(I)+(II)+(III).
\end{eqnarray*}
We have already handled the second term in Step 1. It remains to
treat the first term $(I)$ and the third term $(III).$

Observe that $[X_0,~~ D^\alpha]$ equals to 0 or $D^{\alpha_0}$ for
some $\alpha_0$ with $|\alpha_0|\leq |\alpha|.$  A direct
verification yields
\begin{align*}
   (I) &\leq\|[X_0,~~ \Phi_{\rho,N}]D^\alpha u\|_{n+1}
             +\|\Phi_{\rho,N} D^{\alpha_0}u\|_{n+1}\\
       &\leq \|(D\Phi_{\rho,N})\Phi_{\tilde\rho,N}D^\alpha u\|_{n+1}
             +\|\Phi_{\rho,N} D^{\alpha_0}u\|_{n+1}\\
       &\leq C_n\Big\{~(N/\rho)\|\Phi_{\tilde\rho,N}D^\alpha
       u\|_{n+1}+
             (N/\rho)^{n+2}\|\Phi_{\tilde\rho,N}D^\alpha u\|_{0}
             +\|\Phi_{\rho,N} D^{\alpha_0}u\|_{n+1}~\Big\}.
\end{align*}
For the first term and the third term on the right-hand  side, using
\reff{rho0} in Lemma \ref{r0}, and noting that
$\frac{s\delta}{2}\geq1$, we obtain
\begin{align*}
   &C_n\Big\{~(N/\rho)\|\Phi_{\tilde\rho,N}D^\alpha u\|_{n+1}+
   \|\Phi_{\rho,N} D^{\alpha_0}u\|_{n+1}~\Big\}\\
   &\leq
   C_n\big(N/\rho+1\big)\frac{C_1L^{|\alpha|-2}}{{\tilde\rho}^{(s+n)
   (|\alpha|-3)}}\big((|\alpha|-3)!\big)^{s}
   \\
   &\leq \frac{\widetilde C_{12}L^{|\alpha|-2}}{\rho^{(s+n)(|\alpha|-3)}}
   \big((|\alpha|-3)!\big)^{s}(N/\rho)^{\frac{s\delta}{2}}.
\end{align*}
On the other hand, we use Lemma \ref{tech} with $k=N$ to get
\begin{align*}
  C_n(N/\rho)^{n+2}\|\Phi_{\tilde\rho,N}D^\alpha u\|_{0} \leq
  \frac{\widetilde C_{13}L^{|\alpha|-2}}{\rho^{(s+n)(|\alpha|-3)}}
   \big((|\alpha|-3)!\big)^{s}(N/\rho)^{\frac{s\delta}{2}}.
\end{align*}
Thus
\begin{equation*}\label{+I}
  (I)\leq \frac{\widetilde C_{14}L^{|\alpha|-2}}{\rho^{(s+n)(|\alpha|-3)}}
   \big((|\alpha|-3)!\big)^{s}(N/\rho)^{\frac{s\delta}{2}}.
\end{equation*}

\vspace{0.5ex} Now it remains to eatimate $(III)$. The Leibniz'
formula yields
\begin{align}\label{III1}
\begin{split}
  (III)&\leq \sum\limits_{0<|\gamma|\leq|\alpha|}
        C_\alpha^\gamma
  \big\|\Phi_{\rho,N}(D^{\gamma}a){\widetilde\Lambda}^{2\sigma}D^{\alpha-\gamma}
 u\big\|_{-\frac{\delta}{2}+n+1}\\
  &\leq C_n\sum\limits_{0<|\gamma|\leq|\alpha|}
       C_\alpha^\gamma
       \norm{D^{\gamma}a}_{C^{n+2}(\bar\Omega)}\cdot
       \|\Phi_{\rho,N}{\widetilde\Lambda}^{2\sigma}D^{\alpha-\gamma}
       u\big\|_{-\frac{\delta}{2}+n+1},
\end{split}
\end{align}
where $C_\alpha^\gamma=\frac{\alpha!}{\gamma!(\alpha-\gamma)!}$ are
the binomial coefficients. Since $a\in G^s(\bar\Omega)$, letting
$B_a$ be the Gevrey constant of Gevrey function $a$ on $\bar\Omega$,
we have
\begin{equation}\label{III2}
     \norm{D^{\gamma}a}_{C^{n+2}(\bar\Omega)}
     \leq B_a^{|\gamma|-1}\big((|\gamma|-2)!\big)^{s}~~
     {\rm if}\:\: |\gamma|\geq2, \quad \norm{D^{\gamma}a}_{C^{n+2}(\bar\Omega)}
     \leq B_a~~
     {\rm if}\:\:|\gamma|=0,~1.
\end{equation}
On the other hand, observe that
\begin{align*}
\|\Phi_{\rho,N}{\widetilde\Lambda}^{2\sigma}D^{\alpha-\gamma}u\|_{-\frac{\delta}{2}+n+1}
  \leq &
 \|[{\widetilde\Lambda}^{\sigma},~\Phi_{\rho,N}]
 {\widetilde\Lambda}^{\sigma}D^{\alpha-\gamma} u\|_{-\frac{\delta}{2}+n+1}
    \\
    &+\|\Phi_{\rho,N}{\widetilde\Lambda}^{\sigma}
 D^{\alpha-\gamma}u\|_{(\sigma-\frac{\delta}{2})+n+1}.
\end{align*}
We have handled in Step 1  the first term on the right hand. This
gives
\begin{align*}
  \|[{\widetilde\Lambda}^{\sigma},~\Phi_{\rho,N}]
  {\widetilde\Lambda}^{\sigma}D^{\alpha-\gamma} u\|_{-\frac{\delta}{2}+n+1}
    \leq  \frac{\widetilde C_{15}L^{|\alpha|-|\gamma|-2}}
  {\rho^{(s+n)(|\alpha|-|\gamma|-3)}}\big((|\alpha|-|\gamma|-3)!\big)^{s}
  (N/\rho)^{\frac{s\delta}{2}}.
\end{align*}
For the second term,  note that  $|\alpha|-|\gamma|\leq N-1$ for
$\gamma\neq0.$ We use the induction hypothesis that
 $(E)_{r,N-1}$ holds for all $r\in[0,1],$ to get, for $\gamma,
 0<\abs\gamma\leq\abs\alpha-3,$ that
\begin{align*}
  \|\Phi_{\rho,N}{\widetilde\Lambda}^{\sigma}
  D^{\alpha-\gamma}u\|_{(\sigma-\frac{\delta}{2})+n+1}
  \leq  \frac{L^{|\alpha|-|\gamma|-1}}
  {\rho^{(s+n)(|\alpha|-|\gamma|-3)}}\big((|\alpha|-|\gamma|-3)!\big)^{s}
  (N/\rho)^{s\inner{\sigma-\frac{\delta}{2}}}.
\end{align*}
Observe that
\begin{align*}
  (N/\rho)^{s\inner{\sigma-\frac{\delta}{2}}}\leq  (N/\rho)^{s}&\leq
  \frac{2^s(N-\abs\gamma-2)^s+2^s(\abs\gamma+2)^s}{\rho^s}\\
  &\leq
  16^s(2^s)^{\abs\gamma-1}(N-\abs\gamma-2)^s\rho^{-s},
\end{align*}
Thus  for $\gamma$ with
 $0<\abs\gamma\leq\abs\alpha-3=N-3,$ we have
\begin{align*}
  \|\Phi_{\rho,N}{\widetilde\Lambda}^{\sigma}
  D^{\alpha-\gamma}u\|_{(\sigma-\frac{\delta}{2})+n+1}
  \leq  \frac{16^s(2^s)^{\abs\gamma-1}L^{|\alpha|-|\gamma|-1}}
  {\rho^{(s+n)(|\alpha|-|\gamma|-2)}}\big((|\alpha|-|\gamma|-2)!\big)^{s}
  .
\end{align*}
Note that the above inequality still holds for $\gamma$ with
$\abs\gamma=\abs\alpha-2$ if we take $L\geq
4^{n+1}\inner{\norm{u}_{H^{n+6}(W)}+1}.$  Consequently, we combine
these inequalities to obtain, for $0<\abs\gamma\leq \abs\alpha-2,$
\begin{align*}
  \|\Phi_{\rho,N}{\widetilde\Lambda}^{2\sigma}D^{\alpha-\gamma}u\|_{-\frac{\delta}{2}+n+1}
  \leq \frac{\widetilde
 C_{16}(2^s)^{\abs\gamma-1}L^{|\alpha|-|\gamma|-1}}{\rho^{(s+n)(|\alpha|-|\gamma|-2)}}
    \big((|\alpha|-|\gamma|-2)!\big)^{s}\inner{N/\rho}^{\frac{s\delta}{2}}.
\end{align*}
This together with  \reff{III2} yields
\begin{align*}
   &\sum\limits_{2\leq|\gamma|\leq|\alpha|-2}
      C_\alpha^\gamma
      \norm{D^{\gamma}a}_{C^{n+2}(\bar\Omega)}\cdot\norm{\Phi_{\rho,N}
      {\widetilde\Lambda}^{2\sigma}D^{\alpha-\gamma} u}_{-\frac{\delta}{2}+n+1}\\
 &\leq\nrho^{\frac{s\delta}{2}}\sum\limits_{2\leq|\gamma|\leq|\alpha|-2}
     \frac{\abs\alpha!}{\abs\beta!(\abs\alpha-\abs\beta)!}
     (2^sB_a)^{|\gamma|-1}\big((|\gamma|-2)!\big)^{s}\\
&\hskip 1cm \times     \frac{\widetilde
 C_{16}L^{|\alpha|-|\gamma|}}{\rho^{(s+n)(|\alpha|-|\gamma|-2)}}
     \big((|\alpha|-|\gamma|-2)!\big)^{s}\\
 &\leq\frac{\widetilde
 C_{16}L^{|\alpha|-2}}{\rho^{(s+n)(|\alpha|-3)}}\inner{N/\rho}^{\frac{s\delta}{2}}
     \sum\limits_{2\leq|\gamma|\leq|\alpha|-2}
     \inner{{{2^sB_a}\over L}}^{|\gamma|-1}|\alpha|!
     \big((|\alpha|-4)!\big)^{s-1}
     \\
  &\leq\frac{\widetilde
 C_{16}L^{|\alpha|-2}}{\rho^{(s+n)(|\alpha|-3)}}\big((|\alpha|-3)!
 \big)^{s}\inner{N/\rho}^{\frac{s\delta}{2}}
     \sum\limits_{2\leq|\gamma|\leq|\alpha|-2}
     \inner{{{2^sB_a}\over L}}^{|\gamma|-1}
     \frac{\abs{\alpha}^{3}}{(|\alpha|-3)^{s-1}}.
\end{align*}
Observe that  $s-1\geq 3$ and  thus the series  in the last
inequality is bounded from above by a constant depending only on $n$
if we take $L>2^{s+1}B_a.$ Then we get
\begin{align*}
 \sum\limits_{2\leq|\gamma|\leq|\alpha|-2}
       C_\alpha^\gamma
       &\norm{D^{\gamma}a}_{C^{n+2}(\bar\Omega)}\cdot\norm{\Phi_{\rho,N}
       {\widetilde\Lambda}^{2\sigma}D^{\alpha-\gamma}
       u}_{-\frac{\delta}{2}+n+1}\\
 & \leq\frac{\widetilde C_{17}L^{|\alpha|-2}}{\rho^{(s+n)(|\alpha|-3)}}
       \big((|\alpha|-3)!\big)^{s}(N/\rho)^{\frac{s\delta}{2}}.
\end{align*}
For $|\gamma|=1, ~|\alpha|-1$ or $|\alpha|$, we can compute directly
    \begin{align*}
     \sum\limits_{|\gamma|=1,|\alpha|-1,\abs\alpha}C_\alpha^\gamma
       &\norm{D^{\gamma}a}_{C^{n+2}(\bar\Omega)}\cdot\norm{\Phi_{\rho,N}
       {\widetilde\Lambda}^{2\sigma}D^{\alpha-\gamma}
       u}_{-\frac{\delta}{2}+n+1}\\
      & \leq\frac{\widetilde C_{18}L^{|\alpha|-2}}{\rho^{(s+n)(|\alpha|-3)}}
       \big((|\alpha|-3)!\big)^{s}(N/\rho)^{\frac{s\delta}{2}}.
\end{align*}
Combination of the above two inequalities and \reff{III1} gives that
  \[ (III)\leq \frac{\widetilde
 C_{19}L^{|\alpha|-2}}{\rho^{(s+n)(|\alpha|-3)}}
     \big((|\alpha|-3)!\big)^{s}(N/\rho)^{\frac{s\delta}{2}}.
  \]
Consequently, the desired inequality \reff{step1} follows. This
completes the proof of Step 2.

\bigbreak {\bf Step 3.} In this  step, we prove that if $L\geq
\tilde B$ with $\tilde B$ the Gevrey constant  of  function $f,$
\begin{equation}
\label{step3}
 \|\p\Phi_{\rho,N}D^\alpha u\|_{-\frac{\delta}{2}+n+1}
 \leq\frac{\widetilde C_{20}L^{|\alpha|-2}}{\rho^{(s+n)(|\alpha|-3)}}
      \big((|\alpha|-3)!\big)^{s}(N/\rho)^{\frac{s\delta}{2}}.
\end{equation}

Indeed, observe that
\begin{align*}
\|\p\Phi_{\rho,N}D^\alpha u\|_{-\frac{\delta}{2}+n+1} \leq
\|[\p,~\Phi_{\rho,N}D^\alpha]
u\|_{-\frac{\delta}{2}+n+1}+\|\Phi_{\rho,N}D^\alpha\p
u\|_{-\frac{\delta}{2}+n+1}.
\end{align*}
Since $\p u =f \in G^s(\bar\Omega)$, then $\norm{D^\gamma \p
f}_{H^{n+2}(\Omega)}\leq \tilde
 B$ if $\abs\gamma<n+5,$ and
\[
 \norm{D^\gamma \p f}_{H^{n+2}(\Omega)}\leq \tilde
 B^{\abs\gamma-2}\inner{(\abs\gamma-n-5)!}^s, ~{\rm if ~}
 \abs\gamma\geq n+5.
\]
Hence,
\[
   \|\Phi_{\rho,N}D^\alpha\p u\|_{-\frac{\delta}{2}+n+1}\leq
   C_n (N/\rho)^{n+2}\norm{D^\alpha \p f}_{H^{n+2}(\Omega)}
     \leq\frac{\widetilde C_{21}\tilde B^{|\alpha|-2}}{\rho^{(s+n)(|\alpha|-3)}}
      \big((|\alpha|-3)!\big)^{s}.
\]
We take $L$ such that $L>\tilde B.$ Then the above inequality
together with \reff{step1} in Step 2 yields immediately the
inequality \reff{step3}.

\bigbreak {\bf Step 4.} In the last step we show \reff{r-theta}. And
hence the proof of Lemma \ref{r1/3} will be complete.

\smallskip

First we apply the subelliptic estimate (\ref{sub0}) to get
\[
  \|\Phi_{\rho,N}D^\alpha u\|_{\frac{\delta}{2}+n+1}
  \leq C(\Omega)\bigset{\|\p\Phi_{\rho,N}D^\alpha u\|_{-\frac{\delta}{2}+n+1}
    +\|\Phi_{\rho,N}D^\alpha u\|_{n+1}}
\]
with $C(\Omega)$ a constant depending only on the set $\Omega.$
Combining Lemma \ref{r0} with (\ref{step3}) in Step 3, we have
\begin{eqnarray}\label{N}
\|\Phi_{\rho,N}D^\alpha u\|_{\frac{\delta}{2}+n+1}\leq
\frac{\widetilde
C_{22}L^{|\alpha|-2}}{\rho^{(s+n)(|\alpha|-3)}}\big((|\alpha|-3)!\big)^{s}
(N/\rho)^{\frac{s\delta}{2}} .
\end{eqnarray}
Next, we prove
\begin{equation}\label{finall}
\|\Phi_{\rho,N}{\widetilde\Lambda}^{\sigma} D^\alpha
u\|_{\frac{\delta}{2}-\frac{\delta}{2}+n+1}\leq \frac{\widetilde
C_{23}L^{|\alpha|-2}}{\rho^{(s+n)(|\alpha|-3)}}\big((|\alpha|-3)!\big)^{s}
(N/\rho)^{\frac{s\delta}{2}} .
\end{equation}
Observe that
  \[\|\Phi_{\rho,N}{\widetilde\Lambda}^{\sigma} D^\alpha
 u\|_{\frac{\delta}{2}-\frac{\delta}{2}+n+1}
     \leq \|[{\widetilde\Lambda}^{\sigma},~\Phi_{\rho,N}] D^\alpha
 u\|_{n+1}
       +\|{\widetilde\Lambda}^{\sigma} \Phi_{\rho,N}D^\alpha
 u\|_{n+1}.\]
By the same method as that in Step 1, we get the estimate on the
first term of the right side, that is,
  \[
  \|[{\widetilde\Lambda}^{\sigma},~\Phi_{\rho,N}] D^\alpha u\|_{n+1}
       \leq
\frac{\widetilde
C_{24}L^{|\alpha|-2}}{\rho^{(s+n)(|\alpha|-3)}}\big((|\alpha|-3)!\big)^{s}
(N/\rho)^{\frac{s\delta}{2}} .
\]
Then it remains  to estimate the second term. A direct calculation
gives that
\begin{align*}
  &\|{\widetilde\Lambda}^{\sigma}\Phi_{\rho,N}D^\alpha u\|_{n+1}^2\\
  &=\RE \big(\p \Phi_{\rho,N}D^\alpha u,~
      a^{-1}\Lambda^{2n+2}\Phi_{\rho,N}D^\alpha u\big)
      -\RE\big(\widetilde X_0\Phi_{\rho,N}D^\alpha u,~
      a^{-1}\Lambda^{2n+2}\Phi_{\rho,N}D^\alpha
 u\big)\\
  &={\rm Re}\big(\p\Phi_{\rho,N}D^\alpha u,
    ~a^{-1}\Lambda^{2n+2}\Phi_{\rho,N}D^\alpha u\big)
    -{\frac 1 2}\big(\Phi_{\rho,N}D^\alpha
    u,~[\Lambda^{2n+2},~a^{-1}]\widetilde X_0
    \Phi_{\rho,N}D^\alpha u\big)\\
    &\indent-{\frac 1 2}\big(\Phi_{\rho,N}D^\alpha
 u,~[a^{-1}\Lambda^{2n+2},~\widetilde X_0]
    \Phi_{\rho,N}D^\alpha u\big)\\
  &\leq \widetilde C_{25}\big\{~ \|\p\Phi_{\rho,N}D^\alpha
 u\|_{-\frac{\delta}{2}+n+1}^2
    +\|\Phi_{\rho,N}D^\alpha u\|_{\frac{\delta}{2}+n+1}^2~\big\}.
\end{align*}
This along with  \reff {step3} and \reff{N} shows at once
 \[ \norm{{\widetilde\Lambda}^{\sigma}\Phi_{\rho,N}D^\alpha u}_{n+1}
    \leq\frac{\widetilde C_{26}L^{|\alpha|-2}}{\rho^{(s+n)(|\alpha|-3)}}
    \big((|\alpha|-3)!\big)^{s}(N/\rho)^{\frac{s\delta}{2}}.\]
and hence \reff{finall} follows if we choose $\widetilde
C_{23}=\widetilde C_{24}+\widetilde C_{26}.$ Now by \reff{N} and
\reff{finall}, we obtain the desired inequality \reff{r-theta} if we
choose $\widetilde C_0=\widetilde C_{22}+\widetilde C_{23}$. This
completes the proof of Step 4.

\end{proof}

In  quite  the similar way as that in the proof of Lemma \ref{r1/3},
we can prove by induction the following

\begin{lem}\label{r1}
Assume that $(E)_{r, N-1}$ is true for any $r\in[0, 1]$, then there
exists a constant $C_3,$ depending only on $\sigma, s, n, \norm
u_{H^{n+6}(W)}$ and $\norm{a}_{C^{2n+2}(\Omega)},$  such that for
any $r\in[{\delta\over2},\, \delta],$ if $ L\geq
\max\Big\{2^{s+1}B_a,\, B_f,
\,4^{s+1}\norm{u}_{H^{n+6}(W)}+4^{s+1}\Big\},$ we have, for all
$\alpha, \:\abs\alpha=N,$
\[
  \|\Phi_{\rho,N}D^\alpha
  u\|_{r+n+1}+\|\Phi_{\rho,N}{\widetilde\Lambda}^{\sigma} D^\alpha
  u\|_{r-{\delta\over2}+n+1}\leq
  \frac{C_{3}L^{|\alpha|-2}}{\rho^{(s+n)(|\alpha|-3)}}
  \big((|\alpha|-3)!\big)^{s}(N/\rho)^{sr}.
\]
Inductively, For any $m\in\Natural$ such that
${{m\delta}\over2}<1+{\delta\over2},$ the above inequality still
holds for any $r$ with $\frac{(m-1)\delta}{2}\leq r\leq
\frac{m\delta}{2},$ and hence for all $r$ with $0\leq r\leq 1.$

\end{lem}

\bigbreak Recall that the constants $C_1, C_2, C_3$   in Lemma
\ref{r0}, Lemma \ref{r1/3} and Lemma \ref{r1} depend only on $s,
\sigma, n, \norm{u}_{H^{n+6}(W)}$ and $\norm{a}_{C^{2n+2}(\Omega)}$.
Now take $L$ in such a way that $L>\max\Big\{C_1,\,C_2,\,C_3,\,
2^{s+1}B_a, \,B_f,\, 4^{s+1}(\norm{u}_{H^{n+6}(W)}+1)\Big\}$. Then
by the above three Lemmas, we get the truth of $(E)_{r,N}$ for any
$r\in[0,~1].$ This complete the proof of Proposition \ref{prp4}.


\section{Gevrey regularity of nonlinear equation}
\label{sect4} \setcounter{equation}{0}

In this section, $\mc C_j, j\geq 4,$ will be used to denote suitable
constants depending only on $\sigma,$ the Gevrey index $s$, the
dimension $n$ and the Gevrey constants of the functions $a, F$. The
existence and the Sobolev regularity of weak solutions for
non-linear Cauchy problems was proved in \cite{MX}.  Now let $u\in
L^{\infty}_{loc}(\bb R^{2n+1})$ be a weak solution of (\ref{++1.2}).
We first prove $u\in C^\infty(\bb R^{2n+1}),$ and we need the
following stability results by nonlinear composition (see for
example \cite{Xu}).

\begin{lem}\label{composition}
Let $F(t,x,v,q)\in C^\infty(\bb R^{2n+1}\times\bb R)$ and $r\geq 0$.
If $u\in L^{\infty}_{loc}(\bb R^{2n+1})\cap H_{loc}^{r}(\bb
R^{2n+1})$, then $F(\cdot,u(\cdot))\in H_{loc}^{r}(\bb R^{2n+1}).$
\end{lem}
In fact, if $u_1, u_2\in H^{r}(\bb R^{2n+1}) \cap L^{\infty}(\bb
R^{2n+1})$, then
\[
\|u_1 u_2\|_{r}\leq C_n\{ \|u_1\|_{L^{\infty}}\|u_2\|_{r} +
\|u_2\|_{L^{\infty}}\|u_1\|_{r}\}.
\]
Thus if $r>(2n+1)/2$, the Sobolev embedding theorem implies that
\begin{equation}\label{4.1}
\|u_1 u_2\|_{r}\leq C \|u_1\|_{r}\|u_2\|_{r}.
\end{equation}

Suppose that  $u\in L^{\infty}_{loc}(\bb R^{2n+1})$ is a weak
solution of (\ref{++1.2}). Then by the subelliptic estimate
(\ref{sub0}), one has
\begin{eqnarray}\label{+estimate}
\|\psi_1u\|_{r+\delta}\leq C \{~\|\psi_2
F(\cdot,u(\cdot))\|_r+\|\psi_2 u\|_r~\},
\end{eqnarray}
where $\psi_1, \psi_2 \in C_0^{\infty}(\bb R^{2n+1})$ and $\psi_2=1
$ in the support of $\psi_1$. Combining Lemma \ref{composition} and
the above subelliptic estimate (\ref{+estimate}), we have $u\in
H_{loc}^{\infty}(\bb R^{2n+1})$ by standard iteration. We state this
result in the following Proposition:

\smallskip

\begin{prp}\label{+smooth}

Let $u\in L^{\infty}_{loc}(\bb R^{2n+1})$ be a weak solution of
(\ref{++1.2}). Then $u\in C^\infty(\bb R^{2n+1})$.

\end{prp}

\smallskip

In this section we keep the same notations that we have set up in
the previous sections. We prove the Gevrey regularity of the smooth
solution $u$ of Equation (\ref{++1.2}) on $\Omega$. Set
$W=2\Omega=\big\{(t,x);\:\inner{t^2+\abs
x^2}^{1/2}<2\big\}\times\set{v\in\Real^n,\:\abs v<2}$ and
$$
M=\max_{(t, x, v)\in{\bar {W}}}|u(t, x, v)|.
$$
Let $\{M_j\}$ be a sequence of positive coefficients. We say that it
satisfies the monotonicity condition if there exists $B_0>0$ such
that for any $j\in\bb N$,
\begin{equation}\label{monotonicity}
\frac{j!}{i!(j-i)!}M_iM_{j-i}\leq B_0M_j, ~~(i=1,2,\cdots,j).
\end{equation}
Let $\norm{u}_{C^k(\Omega)}$ be the classic H\"{o}rder norm, that
is, $\norm{u}_{C^k(\Omega)}=\sum_{j=0}^{k}\norm{D^j
u}_{L^\infty(\Omega)}.$

We study now the stability of the Gevrey regularity by the non
linear composition, which is an analogue of Lemma 1 in Friedman's
work \cite{Friedman}.

\begin{lem}\label{+Friedman}

Let $N>n+2$ and $0<\rho<1$ be given.  Let $\{M_j\}$ be a positive
sequence satisfying the monotonicity condition \reff{monotonicity}
and that for some constant $\mc C_n$ depending only on $n,$
\begin{equation}\label{++}
  \nrho^{n+2}M_{N-n-2}\leq \mc C_n M_{N-2};\qquad  M_{j}\geq \rho^{-j},\quad j\geq 2.
\end{equation}
Suppose that there exists $\mc C_4>1,$ depending only on the Gevrey
constant of $F,$ such that:

1) the function $F(t, x, v; q)$ satisfies the following condition:
for any $k, l\geq 2$,
\begin{eqnarray}\label{condition1}
\big\|D_{t,x,v}^{\gamma}D_q^lF\big\|_{C^{n+2}(\bar{\Omega}\times
[-M, M])}\leq \mc C_4^{k+l} M_{k-2}M_{l-2},\quad \forall
~\abs\gamma=k,
\end{eqnarray}

2) the smooth function $g(t, x, v)$ satisfies the following
conditions: $\norm{g}_{L^\infty(\bar W)}\leq M$ and
\begin{eqnarray}\label{condition2}
\|D^j g\|_{C^{n+3}\inner{\bar W}}\leq H_0, \quad 0\leq j\leq 1,
\end{eqnarray}
and for any $0<\rho<1$ and any $j, 2\leq j\leq N,$ one has
\begin{eqnarray}\label{condition3}
\|\Phi_{\rho,j}D^\gamma g\|_{\nu+n+1}\leq H_1^{j-2}M_{j-2}, \quad
\forall ~ \abs \gamma=j,
\end{eqnarray}
where  $\nu$ is a real number satisfying $-1/2<\nu\leq 1,$ and
$H_0,H_1\geq1, \:H_1\geq \inner{4^{n+2}\mc C_4H_0}^2.$

Then there exists $\mc C_5>1,$ depending only the Gevrey constant of
$F$ and the dimension $n,$ such that for all $\rho, 0<\rho<1,$ and
all $\alpha\in\bb N^{2n+1}$ with $|\alpha|=N$,
\begin{eqnarray}\label{conclusion}
\big\|\Phi_{\rho,N}D^{\alpha}\big[F\big(\cdot,
g(\cdot)\big)\big]\big\|_{\nu+n+1}\leq \mc C_5
H_0^2H_1^{N-2}M_{N-2}.
\end{eqnarray}

\end{lem}

\begin{proof} In the proof, we use $C_n$ to denote  constants which depend only on
$n$ and may be different in different contexts. In the following,
for each $\rho,$ we always denote
\[
 \tilde\rho=\frac{(N-1)\rho}{N},
 \qquad \tilde{\tilde\rho}=\frac{(N-2)\rho}{N}.
\]
Observe that for $\rho,\, \tilde\rho, \,\tilde{\tilde\rho},$ we have
the relation \reff{fre}. Since $\Phi_{\tilde\rho, 3}=1$ in the
support of $\Phi_{\rho, N},$ then by Lemma \ref{3.1}, one has
\begin{align*}
  &\bignorm{\Phi_{\rho,N}D^\alpha[F(\cdot,g(\cdot))]}_{\nu+n+1}
  =\bignorm{\Phi_{\rho,N}\Phi_{\tilde\rho, 3}D^\alpha[F(\cdot,g(\cdot))]}_{\nu+n+1}\\
  &\leq C_n\Big\{\bignorm{\Phi_{\tilde\rho, 3}
  D^\alpha[F(\cdot,g(\cdot))]}_{\nu+n+1}+\nrho^{n+1+\nu}\bignorm{\Phi_{\tilde\rho, 3}
  D^\alpha[F(\cdot,g(\cdot))]}_{0}\Big\} \\
  &=:\mc I_1+\mc I_2.
\end{align*}
The proof will be completed if we can show that there exists a
constant $\mc E$ depending only the Gevrey constant of $F$ and the
dimension $n,$ such that
\begin{equation}\label{mc1}
  \mc I_1\leq
  \mc EH_0^2H_1^{N-2}M_{N-2}.
\end{equation}

\vspace{0.5ex} Indeed, choose a multi-index $\tilde\alpha\leq\alpha$
such that $\abs{\tilde\alpha}=n.$ Then
\begin{align*}
  \mc I_2&=C_n
  \nrho^{n+1+\nu}\bignorm{\Phi_{\tilde\rho, 3}D^{\tilde\alpha}\Phi_{\tilde{\tilde\rho}, 3}
  D^{\alpha-\tilde\alpha}[F(\cdot,g(\cdot))]}_{0}\\
  & \leq C_n
  \nrho^{n+1+\nu}\bignorm{\Phi_{\tilde{\tilde\rho}, 3}
  D^{\alpha-\tilde\alpha}[F(\cdot,g(\cdot))]}_{n}\\
  & \leq C_n
  \nrho^{n+2}\bignorm{\Phi_{\tilde{\tilde\rho}, 3}
  D^{\alpha-\tilde\alpha}[F(\cdot,g(\cdot))]}_{\nu+n+1}.
\end{align*}
Assuming that \reff{mc1} holds, then by virtue of the condition
\reff{++}, we have
\begin{align*}
  \mc I_2\leq C_n
  \nrho^{n+2}\bignorm{\Phi_{\tilde{\tilde\rho}, 3}
  D^{\alpha-\tilde\alpha}[F(\cdot,g(\cdot))]}_{\nu+n+1}&\leq
  C_n\nrho^{n+2}\mc EH_0^2H^{N-n-2}M_{N-n-2}\\
  &\leq C_n\mc EH_0^2H^{N-2}M_{N-2}.
\end{align*}
With \reff{mc1}, the conclusion follows at once.

\vspace{0.5ex} The rest is devoted to the proof of \reff{mc1}. By
Faa di Bruno' formula,
$\Phi_{\tilde\rho,3}D^\alpha[F(\cdot,g(\cdot))]$ is the linear
combination of terms of the form
\begin{eqnarray}\label{++A}
\Phi_{\tilde\rho,3}\inner{D_{t,x,v}^\beta\partial_q^lF}(\cdot,g(\cdot))\cdot\prod_{j=1}^l
D^{\gamma_j}g,
\end{eqnarray}
where $|\beta|+l\leq |\alpha|$ and $
\gamma_1+\gamma_2+\cdots+\gamma_l=\alpha-\beta,$  and  if
$\gamma_i=0$,  $D^{\gamma_i}g$ doesn't appear in (\ref{++A}).

\vspace{0.5ex} Next we estimate the Sobolev norm of the form
\reff{++A}. Take a function $\Psi\in C_0^\infty(W)$ such that
$\Psi=1$ in $\Omega.$ Note that $n+1+\nu>(2n+1)/2.$ We apply
(\ref{4.1}) to compute
\begin{align*}
  &\|\Phi_{\tilde\rho,3}\inner{D_{t,x,v}^\beta\partial_q^lF}(\cdot,g(\cdot))\cdot\prod_{j=1}^l
D^{\gamma_j}g\|_{\nu+n+1} \\
&\leq\big\|\Phi_{\tilde\rho,3}
  \inner{D_{t,x,v}^\beta\partial_q^lF}(\cdot,g(\cdot))\big\|_{\nu+n+1}
  \cdot\prod_{j=1}^
  {l}
\big\| \Psi_j D^{\gamma_j}g\big\|_{\nu+n+1},
\end{align*}
where  $\Psi_j$ is given by setting $\Psi_j=\Psi$ if
$\abs{\gamma_j}=1,$ and $\Psi_j=\Phi_{\tilde{\tilde\rho},
\abs{\gamma_j}}$ if $\abs{\gamma_j}\geq2.$ Moreover a direct
computation yields
\begin{align*}
&\big\|\Phi_{\tilde\rho,3}
  \inner{D_{t,x,v}^\beta\partial_q^lF}(\cdot,g(\cdot))\big\|_{\nu+n+1}
  \leq\big\|\Phi_{\tilde\rho,3}\inner{D_{t,x,v}^\beta\partial_q^lF}(\cdot,g(\cdot))
 \big\|_{n+2}\\
  & \leq C_n H_0 \set{\sup\abs{D^{n+2}\Phi_{\tilde\rho,3}}
  \cdot\bignorm{D_{t,x,v}^{\beta}\partial_q^{l}F}_{C(\bar{\Omega}\times[-M,M])}
  +\bignorm{D_{t,x,v}^{\beta}\partial_q^{l}F}_{C^{n+2}(\bar{
  \Omega}\times[-M,M])}}\\
  &\leq C_n H_0 \set{\inner{3\over\rho}^{n+2}\bignorm{D_{t,x,v}^{\beta}\partial_q^{l}F}_{C(\bar{
\Omega}\times[-M,M])}
  +\bignorm{D_{t,x,v}^{\beta}\partial_q^{l}F}_{C^{n+2}(\bar{ \Omega}\times[-M,M])}}
\end{align*}
In the last inequality, we have used \reff{cutoff}. Without loss of
generality we may assume $\abs\beta\geq n+2.$ Then we may choose
$\tilde\beta\leq \beta$ such that
$\big|\tilde\beta\big|=\abs\beta-(n+2).$ Thus by \reff{++},
\reff{condition1} and the monotonicity condition
\reff{monotonicity}, one has
\[
  \bignorm{D_{t,x,v}^{\beta}\partial_q^{l}F}_{C^{n+2}(\bar{ \Omega}\times[-M,M])}
  \leq M_{|\beta|-2}M_{l-2},
\]
and
\begin{align*}
  \inner{3\over\rho}^{n+2}\bignorm{D_{t,x,v}^{\beta}\partial_q^{l}F}_{C(\bar{
  \Omega}\times[-M,M])}&\leq\inner{3\over\rho}^{n+2}\bignorm{D_{t,x,v}^{\tilde\beta}
  \partial_q^{l}F}_{C^{n+2}(\bar{\Omega}\times[-M,M])}\\
  &\leq3^{n+2}M_{n+2}M_{|\tilde\beta|-2}M_{l-2}\\
  &\leq3^{n+2}M_{|\beta|-2}M_{l-2}.
\end{align*}
Hence,
\begin{align*}
  \big\|\Phi_{\tilde\rho,3}
  \inner{D_{t,x,v}^\beta\partial_q^lF}(\cdot,g(\cdot))\big\|_{\nu+n+1}
  \leq C_n H_0 M_{|\beta|-2}M_{l-2}.
\end{align*}
Hence
\begin{equation}\label{formnorm}
  \|\Phi_{\tilde\rho,3}\inner{D_{t,x,v}^\beta\partial_q^lF}(\cdot,g(\cdot))\prod_{j=1}^l
D^{\gamma_j}g\|_{\nu+n+1} \leq  C_n H_0 M_{|\beta|-2}M_{l-2}
  \prod_{j=1}^
  {l}
\big\| \Psi_j D^{\gamma_j}g\big\|_{\nu+n+1},
\end{equation}

\vspace{0.5ex} By virtue of (\ref{condition2})-(\ref{condition3})
and (\ref{++A})-\reff{formnorm}, the situation is now similar to
\cite{Friedman}. In fact, we work with the Sobolev norm, and we
shall follow the idea of \cite{Friedman} to prove \reff{mc1}.  First
we define the polynomial functions \dol{w,\,X_1,\,X_2} in
\dol{\Real} as follows:
\[
  w=w(y)=H_0\inner{y+\sum_{j=2}^N \frac{H_1^{j-2}M_{j-2}y^j}{j!}}, \quad y\in \Real;
\]
\[
  X_1(w)=1+\mc C_4w+\sum_{j=2}^N
  \frac{\mc C_4^{j}M_{j-2}w^j}{j!};
\]
\[
  X_2(y)=1+\mc C_4y+\sum_{j=2}^N \frac{\mc C_4^{j}M_{j-2}y^j}{j!},\quad
  y\in\Real.
\]
By the conditions \reff{condition2} and \reff{condition3},  we have
\[
 \big\|\Psi_{j}D ^j
 g\big\|_{\nu+n+1}
 \leq \frac{d^{j}w(y)}{dy^{j}}\Big|_{y=0},\quad \forall~~1\leq j\leq
N;
\]
Define \dol{X(y,w)=X_1(w)X_2(y).} Then by virtue of
\reff{condition1}, it follows
\[
 M_{k-2}M_{l-2}\leq \frac{\partial^{k+l}X(y,w)}{\partial y^{k}\partial w^l}\Big|_{(y,w)=(0,0)}
 ,\quad
  \forall~~2\leq k,\:l\leq N.
\]
By \reff{formnorm} and the above two inequalities, we get that for
all \dol{\alpha, \:\abs\alpha=N,}
\begin{equation*}
   \mc I_1=C_n\bignorm{\Phi_{\tilde\rho, 3}
  D^\alpha[F(\cdot,g(\cdot))]}_{\nu+n+1}\leq C_nH_0\frac{d^N}{dy^N}
  X\inner{y,w(y)}\Big|_{y=0}.
\end{equation*}
Hence, the proof of \reff{mc1} will be complete if we show that,
\begin{equation}\label{+++++a}
 \frac{d^N}{dy^N}
  \biginner{X_1\inner{w(y)}X_2(y)}\Big|_{y=0}\leq 72 \mc C_4H_0H_1^{N-2}M_{N-2}.
\end{equation}

To prove the above inequality, we need to treat
$X_j^{(k)}(0):=\frac{d^k}{dy^k}X_j(y)\Big|_{y=0},\,0\leq k\leq
N,\,j=1,2.$ We say \dol{w(y)\ll h(y)} if the following relation
holds:
\[
  w^{(j)}(0)\leq h^{(j)}(0),\quad 0\leq j\leq N,
\]
Obviously,
\[
 w(y)\ll w(y)=H_0\inner{y+\sum_{j=2}^N
 \frac{H_1^{j-2}M_{j-2}y^j}{j!}}.
\]
 We next prove
\begin{equation}\label{w2}
  w^2(y)\ll 35 H_0^2\inner{y^2+\sum_{j=3}^N
 \frac{H_1^{j-3}M_{j-3}y^j}{(j-1)!}}.
\end{equation}
In fact, direct verification shows that
\[
  w^2(y)=H_0^2\set{y^2+\sum_{j=3}^N\bigcom{\frac{2H_1^{j-3}M_{j-3}}{(j-1)!}+\sum_{i=2}^{j-2}
  \frac{H_1^{j-4}M_{i-2}M_{j-i-2}}{i!(j-i)!}}y^j}+O(y^{N+1}).
\]
Since \dol{\set{M_j}} satisfies the monotonicity condition
\reff{monotonicity},  we compute
\[
  \sum_{i=2}^{j-2}
  \frac{H_1^{j-4}M_{i-2}M_{j-i-2}}{i!(j-i)!}\leq
  \frac{4 H_1^{j-4}M_{j-4}}{(j-4)!j^2}\sum_{i=2}^{j-2}
  \frac{j^2}{i^2(j-i)^2}\leq\frac{32 H_1^{j-3}M_{j-3}}{(j-1)!}.
\]
Combing these, we obtain \reff{w2}.  Inductively, we have the
following relations:
\[
   w^i(y)\ll 35^{i-1} H_0^i\inner{y^{i}+\sum_{j=i+1}^N
 \frac{H_1^{j-i-1}M_{j-i-1}y^j}{(j-i+1)!}},\quad 2\leq i\leq N-1;
\]
\[
  w^N(y)\ll 35^N H_0^Ny^N.
\]
Thus by the definition of \dol{X_1}, it follows that
\begin{align*}
  &X_1(y)=X_1\inner{w(y)}\ll 1+\mc C_4H_0
  y+\inner{H_0M_0/2+35\mc C_4^2M_0H_0^2/2}y^2\\
  &+\sum_{j=3}^N\inner{\frac{H_0H_1^{j-2}M_{j-2}}{j!}+\frac{35^{j-1}\mc C_4^jH_0^jM_{j-2}}{j!}
  +\sum_{i=2}^{j-1}\frac{35^{i-1}\mc C_4^iH_0^iH_1^{j-i-1}M_{i-2}M_{j-i-1}}{i!(j-i+1)!}}y^j.
\end{align*}
This gives
\[
  X_1(0)=1,\quad X_1'(0)\leq \mc C_4H_0,\quad X_1^{(2)}(0)\leq
  H_0M_0+35\mc C_4^2M_0H_0^2,
\]
and moreover for \dol{j\geq3,}
\[
  X_1^{(j)}(0)\leq \mc C_4H_0H_1^{j-2}M_{j-2}+35^{j-1}\mc C_4^jH_0^jM_{j-2}
  +\sum_{i=2}^{j-1}\frac{j!35^{i-1}\mc C_4^iH_0^iH_1^{j-i-1}M_{i-2}M_{j-i-1}}{i!(j-i+1)!}.
\]
Observe that \dol{H_1\geq (35\mc C_4H_0)^2,} and hence
\dol{X_1^{(2)}\leq 2\mc C_4H_0H_1M_0,} and for \dol{j\geq3,}
\begin{align*}
  X_1^{(j)}(0)&\leq 2C_{4}H_0H_1^{j-2}M_{j-2}+
  \frac{4C_{4}(j-2)!H_0H_1^{j-2}M_{j-3}}{(j-3)!}\sum_{i=2}^{j-1}\frac{j^2}{i^2(j-i)^2}\\
&  \leq 6C_{4}H_0H_1^{j-2}M_{j-2} .
\end{align*}
On the other hand, it is clear that
\[
  X_2(0)=1,\quad X_2'(0)\leq \mc C_4,\quad X_2^{(j)}(0)\leq \mc C_4^jM_{j-2},\:\,2\leq j\leq N.
\]
By virtue of the above relations, we have, for \dol{H_1\geq
\inner{35\mc C_4H_0}^2,}
\begin{align*}
  &\frac{d^N}{dy^N}X\inner{y,
  w(y)}\Big|_{y=0}=\sum_{j=0}^N\frac{N!}{j!(N-j)!}X_1^{(j)}(0)X_2^{(N-j)}(0)\\
  &\leq \mc C_4^{N}M_{N-2}+\mc C_4^NN
  H_0M_{N-3}+2N(N-1)\mc C_4^{N-1}H_0H_1M_0M_{N-4}+6\mc C_4^2 H_0H_1^{N-3}M_{N-3}\\
  &\indent+6\mc C_4H_0H_1^{N-2}M_{N-2}+6C_{4}\sum_{j=3}^{N-2}
  \frac{N!}{j!(N-j)!}H_0H_1^{j-2}M_{j-2}
  \mc C_4^{N-j}M_{N-j-2}\\
  &\leq 72C_{4}H_0H_1^{N-2}M_{N-2}.
\end{align*}
This gives \reff{+++++a}, and hence \reff{mc1}. This completes the
proof of Lemma \ref{+Friedman}.
\end{proof}

Now starting from the smooth solution $u$, we prove the Gevrey
regularity result as follows:

\begin{prp}\label{prp4'}

Let $\delta=\max\set{{\sigma\over4},~{\sigma\over2}-{1\over6}},$ and
let $s\geq {2\over\delta}$ be a real number.  Let
$W=2\Omega=\set{(t,x,v);\:(\frac{t}{2},\frac{x}{2},\frac{v}{2})\in\Omega}.$
Suppose that $u\in C^\infty(\bar W)$ is a solution of \reff{++1.2}
where $ a(t,x,v) \in G^s(\bar \Omega), a>0$ and $F(t,x,v,q)\in
G^s(\bar \Omega\times[-M,M]).$ Then there exits a constant $R$ such
that for any $r\in[0,1]$ and any $N\in\bb{N}$, $N\geq4,$
$$
(E)_{r, N}'\quad\quad
\begin{array}{l}   \|\Phi_{\rho,N}D^\alpha
u\|_{r+n+1}+\|\Phi_{\rho,N}{\widetilde\Lambda}^{\sigma} D^\alpha
u\|_{r-{\delta\over2}+n+1}\\
\hskip 1cm \leq
\frac{R^{|\alpha|-1}}{\rho^{(s+n)(|\alpha|-3)}}\big((|\alpha|-3)!\big)^{s}
\inner{\frac{N}{\rho}}^{sr}
\end{array}
$$
holds for all $\alpha,\:~|\alpha|= N$ and all $0<\rho< 1.$ Thus,
$u\in G^s(\Omega).$
\end{prp}
\begin{rmk}
  Here the Gevrey constant $L$ of $u$ is  determined by the
Gevrey constants $B_a$ and $B_F$ of the functions $a, F$, and
depends only on $s, \sigma, n, \norm{u}_{H^{n+6}(W)}$ and
$\norm{a}_{C^{2n+2}(\Omega)}.$
\end{rmk}

\begin{proof}
We  prove the estimate $(E)_{r, N}'$ by induction on $N$.  We shall
follow the same procedure as that in  the proof of Proposition
\ref{prp4}. First, the truth of $(E)_{r,4}'$ can be deduced by the
same argument as that in  the proof of $(E)_{r,4}$ in the previous
section.

Let now $N>4$ and assume that $(E)_{r, N-1}'$ holds for any $r\in[0,
1]$.  We need to prove the truth of $(E)_{r,N}'$ for $0\leq r\leq
1.$  In the following discussion, we fix $N$  and for each
$0<\rho<1,$ define $\tilde\rho={{N-1}\over N}\rho, \:
\tilde{\tilde\rho}={{N-2}\over N}\rho.$  Let $\Phi_{\rho,N}$ be the
cutoff function which satisfies the property \reff {cutoff}.

First,  the same argument as the proof of Lemma \ref{r0} yields
\begin{equation}\label{r:0}
\|\Phi_{\rho,N}D^\alpha
u\|_{n+1}+\|\Phi_{\rho,N}{\widetilde\Lambda}^{\sigma} D^\alpha
u\|_{-{\delta\over2}+n+1}\leq \frac{ C_1
R^{|\alpha|-2}}{\rho^{(s+n)(|\alpha|-3)}}\big((|\alpha|-3)!\big)^{s},\quad
\forall\:\:0<\rho<1.
\end{equation}
Next we prove,  for all $r, 0<r\leq {\delta\over2},$
\begin{equation}\label{r-theta+}
  \|\Phi_{\rho,N}D^\alpha u\|_{r+n+1}+\|\Phi_{\rho,N}{\widetilde\Lambda}^{\sigma}
 D^\alpha u\|_{r-{\delta\over2}+n+1}
  \leq\frac{
 \mc C_{6}R^{|\alpha|-2}}{\rho^{(s+n)(|\alpha|-3)}}\big((|\alpha|-3)!\big)^{s}
     (N/\rho)^{sr}.
\end{equation}
Observe that  we need only to show the above inequality in the case
when $r={\delta\over2},$ that is
\begin{equation}\label{r-theta++}
  \|\Phi_{\rho, N}D^\alpha u\|_{{\delta\over2}+n+1}+\|\Phi_{\rho,
 N}{\widetilde\Lambda}^{\sigma} D^\alpha u\|_{n+1}
  \leq\frac{
 \mc C_{6}R^{|\alpha|-2}}{\rho^{(s+n)(|\alpha|-3)}}\big((|\alpha|-3)!\big)^{s}
     (N/\rho)^{{{s\delta}\over2}},
\end{equation}
and the truth of \reff{r-theta+} for general
$r\in]0,\,{\delta\over2}[$ follows by the interpolation inequality
\reff{interpolation}.

To prove \reff{r-theta++}, we first show the following inequality
\begin{equation}
\label{step3'}
 \|\p\Phi_{\rho,N}D^\alpha u\|_{-{\delta\over2}+n+1}
 \leq\frac{ \mc C_{7}R^{|\alpha|-2}}{\rho^{(s+n)(|\alpha|-3)}}
      \big((|\alpha|-3)!\big)^{s}(N/\rho)^{{{s\delta}\over2}}.
\end{equation}
In fact,
\begin{align*}
  \|\p\Phi_{\rho,N}D^\alpha u\|_{-{\delta\over2}+n+1}
  &\leq \|[\p,~\Phi_{\rho,N}D^\alpha]u\|_{-{\delta\over2}+n+1}
    +\|\Phi_{\rho,N}D^\alpha\p u\|_{-{\delta\over2}+n+1}\\
  &\leq \|[\p,~\Phi_{\rho,N}D^\alpha]u\|_{-{\delta\over2}+n+1}
  +\|\Phi_{\rho,N}D^\alpha[F(\cdot,u(\cdot))]\|_{-{\delta\over2}+n+1}.
\end{align*}
Since there is no nonlinear form involved in the first term  on the
right-hand  side of the above inequality,  the  same argument as in
the proof of \reff{step1} gives that
\begin{equation}\label{abcd+}
   \|[\p,~~ \Phi_{\rho,N}D^\alpha] u\|_{-{\delta\over2}+n+1}
    \leq\frac{\mc C_8 R^{|\alpha|-2}}{\rho^{(s+n)(|\alpha|-3)}}
    \big((|\alpha|-3)!\big)^{s}(N/\rho)^{{{s\delta}\over2}},
\end{equation}
Thus we need only to treat the second term
$\|\Phi_{\rho,N}D^\alpha[F(\cdot,u(\cdot))]\|_{-{\delta\over2}+n+1}$.
The smoothness of $u$ gives
\begin{equation}\label{++++100}
    \|D^ju\|_{ C^{n+3}(\bar W)}\leq
  \|u\|_{ C^{n+5}(\bar W)},
 \qquad 0
    \leq j\leq2,
\end{equation}
and by the induction hypothesis, for any $3\leq j<N$ and  any
$0<\rho<1, $
\begin{equation}\label{100}
\begin{split}
   \|\Phi_{\rho,j} D^\beta u\|_{ -{\delta\over2}+n+1}
   &\leq \| \Phi_{\rho,j}D^\beta u\|_{n+1}
   \leq \frac{
 C_{1}R^{j-2}}{\rho^{(s+n)(j-3)}}\big((j-3)!\big)^{s}\\
   &\leq\frac{
 C_{1}R^{j-2}}{\rho^{(s+n)(j-3)}}\big((j-3)!\big)^{s}
   (j/\rho)^{{{s\delta}\over2}},\qquad \forall ~\beta, \:\abs\beta=j,
\end{split}
\end{equation}
Similarly,  by \reff{r:0},  we have for any $0<\rho<1, $
\begin{equation}\label{++N}
  \|\Phi_{\rho,N} D^\alpha u\|_{ -{\delta\over2}+n+1}
  \leq\frac{
  C_{1}R^{N-2}}{\rho^{(s+n)(N-3)}}\big((N-3)!\big)^{s}
   (N/\rho)^{{{s\delta}\over2}},\qquad \forall ~\alpha,
   \:\abs\alpha=N.
\end{equation}
Since $F\in G^s(\bar \Omega\times [-M, M])$, then
\begin{equation}\label{200}
  \|D_{t,x,v}^k\partial_q^lF\|_{C^{n+2}(\bar \Omega\times
 [-M,
 M])}
  \leq  B_{F}^{k+l}\big((k-3)!\big)^{s}\big((l-3)!\big)^{s},
 \quad k,l\geq3.
\end{equation}
Define $M_j, H_0, H_1$ by setting
  \[
   H_1=R; \:\: H_0=\norm{u}_{C^{n+3}(\bar W)}+1;\:\: M_0=1; \:\:
M_j={{\big((j-1)!\big)^{s}}\over{\rho^{(s+n)(j-1)}}}
((j+2)/\rho)^{{{s\delta}\over2}},
     \:\:j\geq1.
  \]
We can choose $R$ large enough such that $H_1=R\geq (4^{n+1}B_F
H_0)^2.$ Then (\ref{++++100})-(\ref{200}) can be rewritten as
\begin{equation}\label{+a}
  \|D^ju\|_{ C^{n+3}(\bar W)}\leq H_0,\quad
 0\leq
 j\leq1,
\end{equation}
\begin{equation}\label{+b}
  \| \Phi_{\rho,j}D^\gamma u\|_{ -{\delta\over2}+n+1}
  \leq H_0H_1^{j-2}M_{j-2},\quad\forall~0<\rho<1,~\forall~\abs\gamma=j,\quad 2\leq j\leq N,
\end{equation}
\begin{equation}\label{+c}
\|D_{t,x,v}^k\partial_q^lF\|_{C^{n+2}(\bar \Omega\times [-M, M])}
  \leq  B_{F}^{k+l}M_{k-2}M_{l-2},\quad
k,l\geq2.
\end{equation}
For each $j$, note that $s\geq{2\over\delta}.$ Hence we compute
\begin{align}\label{328}
\begin{split}
  \frac{j!}{i!(j-i)!}M_iM_{j-i}
  &=\frac{j!}{i(j-i)}\big((i-1)!\big)^{s-1}\big((j-i-1)!\big)^{s-1}
     \rho^{-(s+n)(i-1)}\rho^{-(s+n)(j-i-1)}\\
  &\indent\times(i+2)^{{{s\delta}\over2}}(j-i+2)^{{{s\delta}\over2}}\rho^{-s\delta}\\
  &\leq{j!}\big((j-2)!\big)^{s-1}\rho^{-(s+n)(j-2)}
(j+2)^{{{s\delta}\over2}}(j+2)^{{{s\delta}\over2}}\rho^{-s\delta}\\
&\leq\frac{j(j+2)^{{{s\delta}\over2}}}{(j-1)^{s-1}}\big((j-1)!\big)^{s}\rho^{-(s+n)(j-1)}
(j+2)^{{{s\delta}\over2}}\rho^{-{{s\delta}\over2}}\rho^{s+n-{{s\delta}\over2}}\\
  &\leq\frac{j(j+2)^{{{s\delta}\over2}}}{(j-1)^{s-1}}M_j\\
  &\leq\widetilde C_s M_j.
\end{split}
\end{align}
In the last inequality we used the fact that $s-1\geq
1+{{s\delta}\over2},$ where $\widetilde C_s $ is a constant
depending only on $s.$  Moreover, it is easy to verify that, $
M_j\geq {\rho}^{-(s+n)(j-1)}\geq
  \rho^{-j}$ for each $j\geq 2,$ and
\begin{align*}
  \nrho^{n+2}M_{N-n-2}&=\nrho^{n+2}{{\big((N-n-3)!\big)^{s}}\over{\rho^{(s+n)(N-n-3)}}}
  ((N-n)/\rho)^{{{s\delta}\over2}}\\
  &\leq \mc C_n {{\big((N-1)!\big)^{s}}\over{\rho^{(s+n)(N-1)}}}
  ((N+2)/\rho)^{{{s\delta}\over2}}=\mc C_n M_{N-2}.
\end{align*}
Thus $\set{M_j}$ satisfies the monotonicity condition
(\ref{monotonicity}) and  the condition \reff{++}. By virtue of
(\ref{+a})-(\ref{328}), we can use Lemma \ref{+Friedman} with $\nu=
-{\delta\over2}>-{\frac 1 2} $ to obtain
\begin{align*}
  \|\Phi_{\rho,N}D^\alpha[F(\cdot,u(\cdot))]\|_{ -{\delta\over2}+n+1}
  &\leq  \mc C_{5}H_0^2H_1^{|\alpha|-2}M_{|\alpha|-2}\\
  &\leq  2\mc C_{5}\inner{1+\norm
 u_{C^{n+3}(\bar W)}^2}\frac{R^{|\alpha|-2}}{\rho^{(s+n)(|\alpha|-3)}}
    \big((|\alpha|-3)!\big)^{s}(N/\rho)^{{{s\delta}\over2}}.
\end{align*}
This along with \reff{abcd+} yields \reff{step3'}, if we choose $\mc
C_7= \mc C_8+2\mc C_{5}\inner{1+\norm u_{C^{n+3}(\bar W)}^2}$. By
virtue of \reff{step3'}, we can repeat the discussion as in Step 4
in the previous section. This gives  \reff{r-theta++}, and hence
\reff{r-theta+}.

Similarly,  we can prove that for any $r$ with ${\delta\over2}\leq
r\leq\delta$,
\[
  \|D^\alpha
  u\|_{r+n+1,\Omega_\rho}+\|{\widetilde\Lambda}^{\sigma} D^\alpha
  u\|_{r-{\delta\over2}+n+1,\Omega_\rho}\leq
  \frac{\mc C_{9}R^{|\alpha|-2}}{\rho^{(s+n)(|\alpha|-3)}}
  \big((|\alpha|-3)!\big)^{s}(N/\rho)^{sr}.
\]
Inductively, for any $m\in\Natural$ with
${m\delta\over2}<1+{\delta\over2},$ the above inequality still holds
for any $r$ with ${{(m-1)\delta}\over2}\leq r\leq
{{m\delta}\over2}.$ Hence, for  $r$ with $0\leq r\leq 1,$  we obtain
the truth of $(E)_{r,N}'$. This completes the proof of Proposition
\ref{prp4'}.
\end{proof}


\end{document}